\documentclass[12pt]{amsart}
\usepackage{amsthm}
\usepackage{amsmath}
\usepackage{amssymb}
\usepackage{verbatim}
\usepackage[mathscr]{eucal}
\allowdisplaybreaks[4]
\begin{document}

\def\alp{\alpha}
\def\bet{\beta}
\def\gam{\gamma}
\def\del{\delta}
\def\eps{\epsilon}
\def\zet{\zeta}
\def\tht{\theta}
\def\iot{\iota}
\def\kap{\kappa}
\def\lam{\lambda}
\def\sig{\sigma}
\def\ome{\omega}
\def\vep{\varepsilon}
\def\vth{\vartheta}
\def\vpi{\varpi}
\def\vrh{\varrho}
\def\vsi{\varsigma}
\def\vph{\varphi}
\def\Gam{\Gamma}
\def\Del{\Delta}
\def\Tht{\Theta}
\def\Lam{\Lambda}
\def\Sig{\Sigma}
\def\Ups{\Upsilon}
\def\Ome{\Omega}
\def\vDe{\varDelta}
\def\vLm{\varLambda}
\def\vTh{\varTheta}
\def\vGm{\varGamma}
\def\vPh{\varPhi}
\def\vPs{\varPsi}

\def\frka{{\mathfrak a}}    \def\frkA{{\mathfrak A}}
\def\frkb{{\mathfrak b}}    \def\frkB{{\mathfrak B}}
\def\frkc{{\mathfrak c}}    \def\frkC{{\mathfrak C}}
\def\frkd{{\mathfrak d}}    \def\frkD{{\mathfrak D}}
\def\frke{{\mathfrak e}}    \def\frkE{{\mathfrak E}}
\def\frkf{{\mathfrak f}}    \def\frkF{{\mathfrak F}}
\def\frkg{{\mathfrak g}}    \def\frkG{{\mathfrak G}}
\def\frkh{{\mathfrak h}}    \def\frkH{{\mathfrak H}}
\def\frki{{\mathfrak i}}    \def\frkI{{\mathfrak I}}
\def\frkj{{\mathfrak j}}    \def\frkJ{{\mathfrak J}}
\def\frkk{{\mathfrak k}}    \def\frkK{{\mathfrak K}}
\def\frkl{{\mathfrak l}}    \def\frkL{{\mathfrak L}}
\def\frkm{{\mathfrak m}}    \def\frkM{{\mathfrak M}}
\def\frkn{{\mathfrak n}}    \def\frkN{{\mathfrak N}}
\def\frko{{\mathfrak o}}    \def\frkO{{\mathfrak O}}
\def\frkp{{\mathfrak p}}    \def\frkP{{\mathfrak P}}
\def\frkq{{\mathfrak q}}    \def\frkQ{{\mathfrak Q}}
\def\frkr{{\mathfrak r}}    \def\frkR{{\mathfrak R}}
\def\frks{{\mathfrak s}}    \def\frkS{{\mathfrak S}}
\def\frkt{{\mathfrak t}}    \def\frkT{{\mathfrak T}}
\def\frku{{\mathfrak u}}    \def\frkU{{\mathfrak U}}
\def\frkv{{\mathfrak v}}    \def\frkV{{\mathfrak V}}
\def\frkw{{\mathfrak w}}    \def\frkW{{\mathfrak W}}
\def\frkx{{\mathfrak x}}    \def\frkX{{\mathfrak X}}
\def\frky{{\mathfrak y}}    \def\frkY{{\mathfrak Y}}
\def\frkz{{\mathfrak z}}    \def\frkZ{{\mathfrak Z}}

\def\cal{\fam2}
\def\cala{{\cal A}}
\def\calb{{\cal B}}
\def\calc{{\cal C}}
\def\cald{{\cal D}}
\def\cale{{\cal E}}
\def\calf{{\cal F}}
\def\calg{{\cal G}}
\def\calh{{\cal H}}
\def\cali{{\cal I}}
\def\calj{{\cal J}}
\def\calk{{\cal K}}
\def\call{{\cal L}}
\def\calm{{\cal M}}
\def\caln{{\cal N}}
\def\calo{{\cal O}}
\def\calp{{\cal P}}
\def\calq{{\cal Q}}
\def\calr{{\cal R}}
\def\cals{{\cal S}}
\def\calt{{\cal T}}
\def\calu{{\cal U}}
\def\calv{{\cal V}}
\def\calw{{\cal W}}
\def\calx{{\cal X}}
\def\caly{{\cal Y}}
\def\calz{{\cal Z}}

\def\AA{{\mathbb A}}
\def\BB{{\mathbb B}}
\def\CC{{\mathbb C}}
\def\DD{{\mathbb D}}
\def\EE{{\mathbb E}}
\def\FF{{\mathbb F}}
\def\GG{{\mathbb G}}
\def\HH{{\mathbb H}}
\def\II{{\mathbb I}}
\def\JJ{{\mathbb J}}
\def\KK{{\mathbb K}}
\def\LL{{\mathbb L}}
\def\MM{{\mathbb M}}
\def\NN{{\mathbb N}}
\def\OO{{\mathbb O}}
\def\PP{{\mathbb P}}
\def\QQ{{\mathbb Q}}
\def\RR{{\mathbb R}}
\def\SS{{\mathbb S}}
\def\TT{{\mathbb T}}
\def\UU{{\mathbb U}}
\def\VV{{\mathbb V}}
\def\WW{{\mathbb W}}
\def\XX{{\mathbb X}}
\def\YY{{\mathbb Y}}
\def\ZZ{{\mathbb Z}}

\def\bfa{{\mathbf a}}    \def\bfA{{\mathbf A}}
\def\bfb{{\mathbf b}}    \def\bfB{{\mathbf B}}
\def\bfc{{\mathbf c}}    \def\bfC{{\mathbf C}}
\def\bfd{{\mathbf d}}    \def\bfD{{\mathbf D}}
\def\bfe{{\mathbf e}}    \def\bfE{{\mathbf E}}
\def\bff{{\mathbf f}}    \def\bfF{{\mathbf F}}
\def\bfg{{\mathbf g}}    \def\bfG{{\mathbf G}}
\def\bfh{{\mathbf h}}    \def\bfH{{\mathbf H}}
\def\bfi{{\mathbf i}}    \def\bfI{{\mathbf I}}
\def\bfj{{\mathbf j}}    \def\bfJ{{\mathbf J}}
\def\bfk{{\mathbf k}}    \def\bfK{{\mathbf K}}
\def\bfl{{\mathbf l}}    \def\bfL{{\mathbf L}}
\def\bfm{{\mathbf m}}    \def\bfM{{\mathbf M}}
\def\bfn{{\mathbf n}}    \def\bfN{{\mathbf N}}
\def\bfo{{\mathbf o}}    \def\bfO{{\mathbf O}}
\def\bfp{{\mathbf p}}    \def\bfP{{\mathbf P}}
\def\bfq{{\mathbf q}}    \def\bfQ{{\mathbf Q}}
\def\bfr{{\mathbf r}}    \def\bfR{{\mathbf R}}
\def\bfs{{\mathbf s}}    \def\bfS{{\mathbf S}}
\def\bft{{\mathbf t}}    \def\bfT{{\mathbf T}}
\def\bfu{{\mathbf u}}    \def\bfU{{\mathbf U}}
\def\bfv{{\mathbf v}}    \def\bfV{{\mathbf V}}
\def\bfw{{\mathbf w}}    \def\bfW{{\mathbf W}}
\def\bfx{{\mathbf x}}    \def\bfX{{\mathbf X}}
\def\bfy{{\mathbf y}}    \def\bfY{{\mathbf Y}}
\def\bfz{{\mathbf z}}    \def\bfZ{{\mathbf Z}}

\def\scra{{\mathscr A}}
\def\scrb{{\mathscr B}}
\def\scrc{{\mathscr C}}
\def\scrd{{\mathscr D}}
\def\scre{{\mathscr E}}
\def\scrf{{\mathscr F}}
\def\scrg{{\mathscr G}}
\def\scrh{{\mathscr H}}
\def\scri{{\mathscr I}}
\def\scrj{{\mathscr J}}
\def\scrk{{\mathscr K}}
\def\scrl{{\mathscr L}}
\def\scrm{{\mathscr M}}
\def\scrn{{\mathscr N}}
\def\scro{{\mathscr O}}
\def\scrp{{\mathscr P}}
\def\scrq{{\mathscr Q}}
\def\scrr{{\mathscr R}}
\def\scrs{{\mathscr S}}
\def\scrt{{\mathscr T}}
\def\scru{{\mathscr U}}
\def\scrv{{\mathscr V}}
\def\scrw{{\mathscr W}}
\def\scrx{{\mathscr X}}
\def\scry{{\mathscr Y}}
\def\scrz{{\mathscr Z}}

\def\phm{\phantom}
\def\ds{\displaystyle }
\def\smallstrut{\vphantom{\vrule height 3pt }}
\def\bdm #1#2#3#4{\left(
\begin{array} {c|c}{\ds{#1}}
 & {\ds{#2}} \\ \hline
{\ds{#3}\vphantom{\ds{#3}^1}} &  {\ds{#4}}
\end{array}
\right)}

\def\GL{\mathrm{GL}}
\def\SL{\mathrm{SL}}
\def\Sp{\mathrm{Sp}}
\def\SU{\mathrm{SU}}
\def\SO{\mathrm{SO}}
\def\Spin{\mathrm{Spin}}
\def\U{\mathrm{U}}
\def\O{\mathrm{O}}
\def\Mat{\mathrm{M}}
\def\Nr{\mathrm{Nr}}
\def\Tr{\mathrm{Tr}}
\def\tr{\mathrm{tr}}
\def\Ad{\mathrm{Ad}}
\def\Nrm{\mathrm{Nrm}}
\def\Prm{\mathrm{Prm}}
\def\Paf{\mathrm{Paf}}
\def\Lif{\mathrm{Lift}}
\def\lift{\mathrm{lift}}
\def\rank{\mathrm{rank}}
\def\new{\mathrm{new}}
\def\cusp{\mathrm{cusp}}
\def\oname{\mathrm}
\def\sgn{\mathrm{sgn}}
\def\diag{\mathrm{diag}}
\def\nab{\nabla}
\def\La{\langle}
\def\Ra{\rangle}

\def\trs{\,^t\!}
\def\iu{\sqrt{-1}}
\def\oo{\hbox{\bf 0}}
\def\ono{\hbox{\bf 1}}
\def\smallcirc{\lower .3em \hbox{\rm\char'27}\!}
\def\AAf{\AA_{\rm f}}
\def\thalf{\textstyle{\frac{1}{2}}}
\def\bsl{\backslash}
\def\wtl{\widetilde}
\def\til{\tilde}
\def\tx{\text}
\def\Chi{\underline{\chi}}
\def\Psi{\underline{\psi}}
\def\oln{\overline}
\def\Ind{\operatorname{Ind}}
\def\ord{\operatorname{ord}}
\def\pal{\partial}
\def\beq{\begin{equation}}
\def\eeq{\end{equation}}
\def\Let{\begin{pmatrix}}
\def\Rit{\end{pmatrix}}
\def\Left{\left(\begin{smallmatrix}}
\def\Right{\end{smallmatrix}\right)}
\def\Leftt{\begin{smallmatrix}}
\def\Rightt{\end{smallmatrix}}
\def\Sum{\displaystyle \sum}
\def\Prod{\displaystyle \prod}
\def\Iot{^{\iota}}
\def\Pr{^{\prime}}
\def\I{\textgt{i}}
\def\J{\textgt{j}}
\def\K{\textgt{k}}

\newcounter{one}
\setcounter{one}{1}
\newcounter{two}
\setcounter{two}{2}
\newcounter{thr}
\setcounter{thr}{3}
\newcounter{fou}
\setcounter{fou}{4}
\newcounter{fiv}
\setcounter{fiv}{5}

\def\FRAC#1#2{\leavevmode\kern.1em
\raise.5ex\hbox{\the\scriptfont0 #1}\kern-.1em
/\kern-.15em\lower.25ex\hbox{\the\scriptfont0 #2}}
\newcommand{\Leg}[2]{\left(\!\FRAC{\ensuremath{#1}}{\ensuremath{#2}}\right)}
\def\DFRAC#1#2{\leavevmode\kern.1em
\raise1.0ex\hbox{#1}\kern-.2em
{\Big/}\kern-.3em\lower.9ex\hbox{#2}}
\newcommand{\DLeg}[2]{\left(\!\DFRAC{\ensuremath{#1}}{\ensuremath{#2}}\right)}
\def\mtrx#1#2#3#4{\left(\begin{smallmatrix} {\text{\small $#1$}}\; & {\text{\small $#2$}} \\ \vphantom{C^{C^C}}{\text{\small $#3$}}\; & \vphantom{C^2}{\text{\small $#4$}} \end{smallmatrix}\right)}

\newcommand{\sfrac}[2]{%
\raisebox{0.2em}{$#1$}\kern-0.1em%
\raisebox{0.1em}{/}\kern-0.1em\raisebox{-0.2em}{$#2$}}

\def\lddots{\mathinner{\mskip1mu\raise1pt\vbox{\kern7pt\hbox{.}}\mskip2mu\raise4pt\hbox{.}\mskip2mu\raise7pt\hbox{.}\mskip1mu}}
\def\onon{\Left 0 & -1 \\ 1 & 0\Right}

\def\today{\ifcase\month\or
 January\or February\or March\or April\or May\or June\or
 July\or August\or September\or October\or November\or December\fi
 \space\number\day, \number\year}

\makeatletter
\@addtoreset{equation}{section}
\def\theequation{\thesection.\arabic{equation}}

\theoremstyle{plain}
\newtheorem{theorem}{Theorem}[section]
\newtheorem*{main_theorem}{Theorem}
\newtheorem{lemma}[theorem]{Lemma}
\newtheorem{proposition}[theorem]{Proposition}
\theoremstyle{definition}
\newtheorem{definition}[theorem]{Definition}
\newtheorem{conjecture}[theorem]{Conjecture}
\theoremstyle{remark}
\newtheorem{remark}[theorem]{Remark}
\newtheorem*{main_remark}{Remark}
\newtheorem{corollary}[theorem]{Corollary}

%%%%%%%%%%%%%%%%%%%%%%%%%%%%%%%%%%%%%%%%%%%%%%%%%%%%%%%%%%%%%%%%%%%%%%%%%

\title[]{Jacobi forms of degree one}
\author{Shunsuke Yamana}
\address{Graduate school of mathematics, Kyoto University, Kitashirakawa, Kyoto, 606-8502, Japan}
\email{yamana07@math.kyoto-u.ac.jp}
%\keywords{Jacobi form; Maass lift}
%\subjclass{11F30,11F50, 11F27}
\begin{abstract}
We show that a certain subspace of space of elliptic cusp forms is isomorphic as a Hecke module to a certain subspace of space of Jacobi cusp forms of degree one with matrix index by constructing an explicit lifting. 
This is a partial generalization of the work of Skoruppa and Zagier. 
This lifting is also related with the Ikeda lifting.  
\end{abstract}
\thanks{The author thanks Prof.~Ikeda, Dr.~Kawamura, Dr.~Narita, Prof.~Sugano, and Prof.~Ueda for useful discussion. The author is supported by JSPS Research Fellowships for Young Scientists. }
\maketitle

%%%%%%%%%%%%%%%%%%%%%%%%%%%%%%%%%%%%%%%%%%%%%%%%%%%%%%%%%%%%%%%%%%%%%%%%%

\section*{\bf Introduction}\label{intro}

The work of Skoruppa and Zagier \cite{SZ} established a bijective correspondence between elliptic modular forms and Jacobi forms of degree one with scalar index which is compatible with the action of the Hecke operators. 
However, it still remains difficult to generalize all of their results to matrix index. 
%produce corresponding Jacobi forms explicitly in terms of Fourier expansion (although a certain family of explicit liftings from Jacobi forms to elliptic modular forms is constructed in \cite{SZ}). It seems more difficult to generalize their result to matrix index. 

In this paper we will give an explicit recipe which associates to an elliptic cusp form a Jacobi cusp form of degree one with matrix index, which gives a partial generalization of \cite{SZ}.   

Let us give a more precise description of our results. 
Throughout this paper, we fix a positive integers $k$, $\kap$ and a positive definite symmetric even integral matrix $S$ of rank $n$ satisfying the following conditions: 
\begin{itemize}
\item $L=\ZZ^n$ is a maximal integral lattice with respect to $S$; 
\item $\kap=k+\bigl[\frac{n+1}{2}\bigl]$ is even. 
\end{itemize}
Let $L^*$ be the dual lattice of $L$ with respect to $S$ and $\calt^+$ the set of all pairs $(a, \alp)\in \ZZ\times L^*$ such that $a>S[\alp]/2$, where $S[\alp]=\trs\alp S\alp$.  

For each prime $p$, the quadratic form $q_p$ on $V_p=L/pL$ is defined by 
\[q_p[x]=S[x]/2\pmod p. \] 
Let $s_p(S)$ be the dimension of the radical of $(V_p,q_p)$. 
Note that $s_p(S)\leq 2$ since $L$ is a maximal integral lattice. 
Put $\frkS_i=\{p\;|\;s_p(S)=i\}$. 
Let $\nu_p$ be the Witt index of $S$ over $\QQ_p$ and define $\eta_p(S)\in\{\pm 1\}$ by 
\[n=2\nu_p+2-\eta_p(S). \]  

For the sake of simplicity, we suppose that $n$ is odd in this introductory section. 
Let $b_S$ (resp. $d_S$) be the product of the rational primes in $\frkS_1$ (resp. $\frkS_2$). 
For $(a,\alp)\in\calt^+$, we put
\begin{align*}
\Del_{a,\alp}&=\det S_{a,\alp}, & S_{a,\alp}&=\mtrx{S}{S\alp}{\trs\alp S}{2a}.  
\end{align*}

We denote by $\frkH$ the upper half-plane and put $\cald=\frkH\times\CC^n$. 
We write $J^\cusp_{\kap,S}$ for the space of Jacobi cusp forms of weight $\kap$ with index $S$. 
Our main object of study is the following subspace: 
\[J^{\cusp,M}_{\kap,S}=\Big\{\phi\in J^\cusp_{\kap,S}\;\Big|\; \phi(\tau,w)=\sum_{(a,\alp)\in\calt^+}c_\phi(\Del_{a,\alp})q^a\bfe(S(\alp,w))\Big\}, \]
where $q=\bfe(\tau)=e^{2\pi\iu\tau}$ and $(\tau,w)\in\cald$. 

For a nonzero element $a$ in $\QQ_p$, put $\Psi_p(a)=1$, $-1$, $0$ according as $\QQ_p(\sqrt{a})$ is $\QQ_p$, an unramified quadratic extension of $\QQ_p$ or a ramified quadratic extension of $\QQ_p$. 
For each positive rational number $N$, we denote the absolute value of the discriminant of $\QQ(((-1)^kN)^{1/2})/\QQ$ by $\frkd_N$. 
Let $\frkf_N$ be the positive rational number such that $N=\frkd_N\frkf_N^2$. 

Let $b$ (resp. $d$) be a positive divisor of $b_S$ (resp. $d_S$). Let 
\[f(\tau)=\sum_{m=1}^\infty c_f(m)q^m\in S_{2k}(\Gam_0(bd))\]
be a primitive form, the $L$-function of which is given by 
\begin{multline*}
L(f,s)=\sum_{m=1}^\infty c_f(m)m^{-s}\\
=\prod_{p|bd}(1-\alp_pp^{k-1/2-s})^{-1}
\prod_{p\nmid bd}(1-\alp_pp^{k-1/2-s})^{-1}(1-\alp_p^{-1}p^{k-1/2-s})^{-1}.
\end{multline*}
Suppose that $f|W(p)=\eta_p(S)f$ for each prime divisor $p$ of $b$, where $W(p)$ is the Atkin-Lehner involution on $S_{2k}(\Gam_0(bd))$.  
Let 
\[g(\tau)=\sum_{m=1}^\infty c_g(m)q^m\in S^+_{k+1/2}(4bd)\]
be a Hecke eigenform which corresponds to $f$ via the Shimura correspondence. 
Here, $S^+_{k+1/2}(4bd)$ is the Kohnen plus subspace on $\Gam_0(4bd)$. 
Note that $g$ is unique up to scalar (see \S \ref{sec:2}). 
For each positive integer $N$, we put 
\[c_\vPh(N)=2^{-\bfb_{bd}(N)}c_g(\frkd_{N})\frkf_{N}^{k-1/2}\prod_pl_{p,S,N}(\alp_p), \]
where $l_{p,S,N}$ is a reciprocal Laurent polynomial given explicitly in \S \ref{sec:3} and 
\[\bfb_{bd}(N)=\sharp\{\tx{prime divisors $p$ of }bd\;|\; \Psi_p((-1)^kN)\neq 0\}. \]   
The space of newforms for $S_{2k}(\Gam_0(bd))$ is denoted by $S^\new_{2k}(bd)$. 

The aim in this paper is to prove the following theorem. 

\begin{main_theorem}[{cf. Theorem \ref{thm:31}}]
If $\kap=k+\frac{n+1}{2}$ is even, then 
\[\vPh(\tau,w)=\sum_{(a,\alp)\in\calt^+}c_\vPh(\Del_{a,\alp})q^a\bfe(S(\alp,w))\]
is a Hecke eigenform in $J^{\cusp,M}_{\kap,S}$ whose $L$-function is given by 
\[L(\vPh,s)=\prod_p(1-\alp_pp^{-s})^{-1}(1-\alp_p^{-1}p^{-s})^{-1}. \] 
Moreover, the lifting $f\mapsto\vPh$ gives a bijective correspondence, up to scalar multiple, between Hecke eigenforms in 
\[\oplus_{b,d\geq 1,\;b|b_S,\; d|d_S}\{f\in S^\new_{2k}(bd)\;|\;f|W(p)=\eta_p(S)f\tx{ for each }p|b\}, \]
and those in $J^{\cusp,M}_{\kap,S}$. 
\end{main_theorem}

We also obtain a similar lifting in the case when $n$ is even (see Theorem \ref{thm:32}). 
While \cite{SZ} uses the trace formula, our method of proof is fairly classical. 
The process of the proof is nothing but a generalization of one of main steps in the proof of the Saito-Kurokawa conjecture exposited in Eichler-Zagier \cite{EZ}. 
 
It turns out that there is an auxiliary space $S^S_{k+1/2}(\vDe_S)$ such that the correspondence above is the composition of three isomorphisms:  
\begin{gather*}
\quad\quad\quad\quad\quad J^{\cusp,M}_{\kap,S}(\simeq S^M_\kap(\Tht))\\
i_1\; |\wr \\
S^S_{k+1/2}(\vDe_S)\\
i_2\; |\wr \\
\oplus_{b,d\geq 1,\;b|b_S,\; d|d_S}\{g\in S^\new_{k+1/2}(4bd))\;|\;g\tx{ satisfies (\roman{one}) for each }p|b\}\\
i_3\; |\wr \\
\oplus_{b,d\geq 1,\;b|b_S,\; d|d_S}\{f\in S^\new_{2k}(bd)\;|\;f|W(p)=\eta_p(S)f\tx{ for each }p|b\}. 
\end{gather*} 
As is well-known, a Jacobi form $\phi\in J^{\cusp,M}_{\kap,S}$ is expressed as a sum 
\[\phi(\tau,w)=\sum_\mu\phi_\mu(\tau)\vth^S_\mu(\tau,w), \]
where 
\[\phi_\mu(\tau)=\sum_mc_\phi(\Del_{m,\mu})\bfe((m-S[\mu]/2)\tau) \]
and $\{\vth^S_\mu\}$ is a collection of theta functions indexed by $\mu\in L^*/L$. 
The first map $i_1$ is a canonical map defined by modifying the mapping $\phi\mapsto\phi_0$ slightly. 
Putting $\vDe_S=4b_Sd_S$ and observing the conditions 
\begin{enumerate}
\renewcommand\labelenumi{(\theenumi)}
\item[(\roman{one})] $c_g(m)=0$ if $\biggl(\dfrac{(-1)^km}{p}\biggl)=-\eta_p(S)$,    
\item[(\roman{two})] $\Tr^S_p(g)=0$   
\end{enumerate} 
for $p\in\frkS_1\cup\frkS_2$, we define the space $S^S_{k+1/2}(\vDe_S)$ by 
\begin{multline*}
S^S_{k+1/2}(\vDe_S)=\{g\in S^+_{k+1/2}(\vDe_S)\;|\; g\tx{ satisfies (\roman{one}) for each }p\in\frkS_1 \\
\tx{and satisfies (\roman{two}) for each }p\in\frkS_2\}. 
\end{multline*}
Here, $\Tr^S_p$ is the trace operator from $S^+_{k+1/2}(\vDe_S)$ to $S^+_{k+1/2}(p^{-1}\vDe_S)$. 
The second canonical isomorphism $i_2$ is proved in \S \ref{sec:4}, and $i_3$ is the Shimura correspondence. 

Recall that the Siegel Eisenstein series played a crucial role in Ikeda's proof of the Duke-Imamoglu conjecture. 
We shall show that the Jacobi Eisenstein series of weight $\kap$ with index $S$ has a Fourier expansion of the form 
\[E_{\kap,S}(\tau,w)=\sum_{a=S[\alp]/2}q^a\bfe(S(\alp,w))
+C_{\kap,S}\sum_{(a,\alp)\in\calt^+}A(\Del_{a,\alp})q^a\bfe(S(\alp,w)), \]
where 
\[A(N)=L(1-k,\psi_{(-1)^kN})\frkf_N^{k-1/2}\prod_pl_{p,S,N}(p^{k-1/2})\]  
for a suitable constant $C_{\kap,S}$ (see Corollary \ref{cor:81}). 
Here, $L(s, \psi_{(-1)^kN})$ is the Dirichlet $L$-function with the primitive Dirichlet character $\psi_{(-1)^kN}$ corresponding to $\QQ(((-1)^kN)^{1/2})/\QQ$. 
Therefore the Fourier coefficients of $\vPh$ and $E_{\kap,S}$ turn out to be quite similar. 
This similarity allows us to prove the surjectivity of $i_1$ and to show that $\vPh$ is a Hecke eigenform. 

We first discuss the case when $n$ is odd in \S \ref{sec:4} and \S \ref{sec:5}, and then the case when $n$ is even in \S \ref{sec:6} and \S \ref{sec:7}. 
A uniform treatment of these two cases is not impossible but cumbersome, so we believe that our treatment, if redundant to some extent, makes our exposition easier to read. 
 
The proof includes the following decomposition
\[J^\cusp_{\kap,S}=J^{\cusp,M}_{\kap,S}\oplus J^{\cusp,0}_{\kap,S}, \]
where 
\[J^{\cusp,0}_{\kap,S}=\{\phi\in J^\cusp_{\kap,S}\;|\;\phi_0\tx{ is identically zero}\} \]
(see Proposition \ref{prop:72}). 
It seems likely that an analogous result holds for the space $J^{\cusp,0}_{\kap,S}$. 
However, as a strategy for attacking this problem, we may need a more general machinery, for example, a trace formula. 

We shall in fact show that  
\[l_{p,S,\Del_{a,\alp}}(X)=\til{F}_p(2^{-1}S_{a,\alp}; X), \]
where the right-hand side is the essential part of the Siegel series for $2^{-1}S_{a,\alp}$ at $p$ (see Proposition \ref{prop:92}). 
It follows that $\vPh$ coincides with the $S/2$-th Fourier-Jacobi coefficient of the Ikeda lift of $f$ if $b=d=1$, and as such, our result in this paper is used in our forthcoming work on a generalization of Maass relations to higher genus. 
This application is one of the main motivations for the present paper. 

Recall that another main step in the proof of the Saito-Kurokawa conjecture is the Fourier-Jacobi expansion of Siegel modular forms in the Maass space. 
This step was studied by many authors and generalized to holomorphic cusp forms on orthogonal groups of signature $(2,n+2)$ (for example, see Kojima \cite{Kj}, Krieg \cite{Kr,Kr2,Kr3}, Gritsenko \cite{G} and Sugano \cite{Su}). 
Murase and Sugano \cite{MS} studied the space $J^{\cusp,M}_{\kap,S}$ in connection with the Maass space $S^M_\kap(\Tht)$ attached to the orthogonal group. 
In \S \ref{sec:10} we restate their works by exploiting our result just described. 

%%%%%%%%%%%%%%%%%%%%%%%%%%%%%%%%%%%%%%%%%%%%%%%%%%%%%%%%%%%%%%%%%%%%%%%%%

\section*{\bf Notation}

Let $X$ be an $n$-dimensional vector space over $\QQ$ which is equipped with a positive definite quadratic form $S: X\to \QQ$. The associated bilinear form is given by
\[S(x, y)=(S[x+y]-S[x]-S[y])/2. \]
Let $L$ be a maximal integral lattice with respect to $S$, namely, if $M$ is a lattice containing $L$ such that $S[x]/2\in\ZZ$ for every $x\in M$, then $M=L$. 
We frequently identify $S$ (resp. $X$, $L$) with an even integral positive definite symmetric matrix of rank $n$ (resp. $\QQ^n$, $\ZZ^n$). 

For $x\in\RR$, we denote by $[x]$ the Gauss bracket of $x$. 
We denote by the formal symbol $\infty$ the infinite place of $\QQ$ and do not use $p$ or $q$ for the infinite place. 
Let $\AA$ be the adele ring of $\QQ$ and $\AAf$ the finite part of the adele ring. 

Set $q=\bfe(\tau)=\bfe_\infty(\tau)=\exp(2\pi\iu\tau)$ for $\tau\in\CC$. 
For a modular form $h$, we denote the $m$-th Fourier coefficient of $h$ by $c_h(m)$. 
Put 
\[\bfe_p(x)=\exp(-2\pi\iu(\tx{fractional part of }x))\] 
for $x\in \QQ_p$. 
If $z\in\CC$ and $\ell\in\ZZ$, then $z^{1/2}$ denotes the square root of $z$ such that $-\pi/2<\arg z^{1/2}\leq \pi/2$ and $z^{\ell/2}=(z^{1/2})^\ell$. 

Let $\ono_m$ (resp. $\oo_m$) be the identity (resp. zero) matrix of degree $m$. 
If $a_1, \dots, a_m$ are square matrices, $\diag[a_1, \dots, a_m]$ denotes the matrix with $a_1, \dots, a_m$ in the diagonal blocks and $0$ in all other blocks. 
The symbol $\del((*))$ stands for either $1$ or $0$ according to the condition $(*)$ is satisfied or not. 
For an algebraic group $G$ over $\QQ$, the group of $D$-valued points, where $D$ is any $\QQ$-algebra, is denoted by $G(D)$. 

%%%%%%%%%%%%%%%%%%%%%%%%%%%%%%%%%%%%%%%%%%%%%%%%%%%%%%%%%%%%%%%%%%%%%%%%%

\section{\bf Preliminaries on Jacobi forms}\label{sec:1}

We discuss the notion of Jacobi forms of degree one in this section. 
We fix once and for all a positive definite symmetric even integral matrix $S$ of degree $n$. 
We assume that $S$ is maximal, i.e., there is no non-degenerate matrix $\alp\in\Mat_n(\ZZ)$ such that $\det\alp>1$ and $S[\alp^{-1}]$ is even integral. 
The Heisenberg group $H_S$ is an algebraic group over $\ZZ$, the group of $\calo$-valued points of which is given by
\[H_S(\calo)=\{[\xi, \eta, \zet]\;|\;\xi,\;\eta\in L\otimes_\ZZ\calo,\; \zet\in\calo\}\]
for any ring $\calo$ and the composition rule of which is given by
\[[\xi, \eta, \zet]\cdot[\xi\Pr, \eta\Pr, \zet\Pr]=[\xi+\xi\Pr, \eta+\eta\Pr, \zet+\zet\Pr+S(\xi, \eta\Pr)]. \]

The special linear group $\SL_2$ acts on $H_S$ by 
\[\alp^{-1}[\xi, \eta, \zet]\alp=[\xi\Pr, \eta\Pr, \zet\Pr], \]
where
\begin{align*}
(\xi\Pr, \eta\Pr)&=(\xi, \eta)\alp, & \zet\Pr&=\zet-2^{-1}S(\xi, \eta)+2^{-1}S(\xi\Pr, \eta\Pr) 
\end{align*}
for $\alp\in\SL_2$. 
We thus obtain an algebraic group $J_S=\SL_2\cdot H_S$, which is called the Jacobi group. 
We write $\frkH$ for the upper half-plane. 
The archimedean part $J_S(\RR)$ acts transitively on $\cald=\frkH\times\CC^n$ by 
\[[\xi,\eta,\zet]\alp(\tau,w)=(\alp\tau,w(c\tau+d)^{-1}+\xi\cdot\alp\tau+\eta), \]
where $\alp\tau=(a\tau+b)(c\tau+d)^{-1}$ for $\alp=\Left a & b \\ c & d\Right\in\SL_2(\RR)$. 
We define the automorphy factor $j_\kap$ on $J_S(\RR)\times\cald$ by 
\begin{multline*}
j_\kap([\xi,\eta,\zet]\alp,(\tau,w))\\=(c\tau+d)^\kap
\bfe\biggl(-\zet+\frac{cS[w]}{2(c\tau+d)}-\frac{S(\xi,w)}{c\tau+d}-\frac{S[\xi]\cdot\alp\tau}{2}\biggl). 
\end{multline*}

For a function $\phi:\cald\to\CC$, $\kap\in\ZZ$ and $\gam\in J_S(\RR)$, we define $\phi|_\kap\gam:\cald\to\CC$ by  
\[\phi|_\kap\gam(\tau,w)=j_\kap(\gam,(\tau,w))^{-1}\phi(\gam(\tau,w)). \]

Put $\vGm=\SL_2(\ZZ)H_S(\ZZ)$ and $L^*=S^{-1}L$. Let $\calt^+$ (resp. $\calt^0$) be the set of all pairs $(a, \alp)\in \ZZ\times L^*$ such that $a>S[\alp]/2$ (resp. $a=S[\alp]/2$). 

A Jacobi (resp. Jacobi cusp) form $\phi$ of weight $\kap$ with index $S$ is a holomorphic function on $\cald$ which satisfies $\phi|_\kap\gam=\phi$ for every $\gam\in\vGm$ and has a Fourier expansion of the form 
\[\phi(\tau, w)=\sum_{(a, \alp)}c_\phi(a, \alp)q^a\bfe(S(\alp, w)), \]
where $(a, \alp)$ extends over all the pairs of $\calt^0\cup\calt^+$ (resp. $\calt^+$). 

The space of Jacobi (resp. Jacobi cusp) forms of weight $\kap$ with index $S$ is denoted by 
$J_{\kap,S}$ (resp. $J^\cusp_{\kap,S}$). 

We write $\FF_q$ for a finite field with $q$ elements. Let $V_p=\FF_p^n$ and define the quadratic form $q_p$ on $V_p$ by $q_p[x]=S[x]/2\pmod p$. 
The radical $\mathrm{Rad}(V_p)$ of $(V_p, q_p)$ is the $\FF_p$-vector space which consists of all elements $x\in V_p$ such that $q_p[x]=0$ and $q_p[x+y]=q_p[x]+q_p[y]$ for every $y\in V_p$. 
Put $s_p(S)=\dim_{\FF_p}\mathrm{Rad}(V_p)$. 
Since $S$ is maximal, we have $s_p(S)\leq 2$ (cf. Lemma \ref{lem:52}, \ref{lem:71} and their proofs). We write $\frkS_i$ for the set of all rational primes $p$ such that $s_p(S)=i$.  
We denote by $d_S$ the product of prime numbers in $\frkS_2$. 

If $n$ is odd, then we put 
\begin{align*}
b_S&=\prod_{p\in\frkS_1}p, &\Lam_S&=\prod_{2\neq p\in\frkS_2}p, & 
\del_S&=\begin{cases} d_S &\tx{if $2\notin\frkS_2$. }\\
2d_S &\tx{if $2\in\frkS_2$. }
\end{cases} 
\end{align*}

Whenever $n$ is even, let $K$ be the discriminant field of $S$, i.e., the extension of $\QQ$ 
\[K=\QQ(((-1)^{n/2}\det S)^{1/2}). \]
We denote by $\frkd_S$ for the absolute value of discriminant of $K/\QQ$ and by $\chi_S$ the primitive Dirichlet character corresponding to $K/\QQ$. 
It is worth noting that $\frkS_1$ coincides with the set of prime factors of $\frkd_S$. 
We put $\xi_p(S)=\chi_S(p)$ and denote by $\Chi_S=\otimes_v\Chi_{S, v}$ the character of $\AA^\times/\QQ^\times$ corresponding to $K$. 

Put 
\[D_S=\begin{cases}
4b_S\Lam_S &\tx{if $2\nmid n$. }\\
\frkd_Sd_S &\tx{if $2|n$. }
\end{cases}\]
An easy relation   
\beq
D_S=\begin{cases}
2\del_S^{-1}\det S &\tx{if $2\nmid n$ }\\
d_S^{-1}\det S &\tx{if $2|n$ }
\end{cases}\label{tag:11}
\eeq
follows from the classification of maximal $\ZZ_p$-integral lattices (cf. \cite[\S 9]{E}). 
We put
\[D_{a,\alp}=D_S(a-S[\alp]/2) \]
for $(a,\alp)\in\ZZ\times L^*$. 
When $n$ is odd, we put 
\begin{align*}
\Del_{a,\alp}&=\det S_{a,\alp}, & S_{a,\alp}&=\mtrx{S}{S\alp}{\trs\alp S}{2a}. 
\end{align*}

Let $p$ be a rational prime, $(\;,\;)_p$ the Hilbert symbol over $\QQ_p$ and $B$ a non-degenerate symmetric matrix of size $r$ with entries in $\QQ_p$. 
Recall that the isometry classes of non-degenerate quadratic spaces over $\QQ_p$ are characterized by their dimension, determinant and Hasse invariant. 
Note that $\det B\in \QQ_p^\times/\QQ_p^{\times2}$ and the character
\[\Chi_B(x)=(x,(-1)^{r(r+1)/2}\det B)_p\]
determine each other uniquely. 
If $r>2$ and a quadratic character $\Chi$ are fixed, then there are precisely two quadratic forms $B$ of rank $r$ with $\Chi_B=\Chi$, having opposite Hasse invariants. 
When $r=2$, there are again two forms if $\Chi\neq 1$ and only one (the split form) if $\Chi=1$. 

\begin{definition}\label{def:11}
We normalize our Hasse invariant $\eta_p(B)$ so that it depends only on the isometry class of an anisotropic kernel of $B$. 
More precisely, we can define $\eta_p(B)$ by the following conditions:  
\begin{enumerate}
\renewcommand\labelenumi{(\theenumi)}
\item[(\roman{one})] If $r$ is odd, then an anisotropic kernel of $B$ has dimension $2-\eta_p(B)$. 
\item[(\roman{two})] If $r$ is even and $\Chi_B\neq 1$, and if we choose an element $\alp\in\QQ_p^\times$ such that $\Chi_B(\alp)=\eta_p(B)$, then $B$ is the orthogonal sum of a split form of dimension $r-2$ with a norm form scaled by the factor $\alp$ on the quadratic extension $\QQ_p(((-1)^{r/2}\det B)^{1/2})$. 
\item[(\roman{thr})] When $r$ is even and $\Chi_B=1$, then $B$ is split or quaternionic according as $\eta_p(B)=1$ or $\eta_p(B)=-1$. 
We here call $B$ quaternionic if $B$ is the orthogonal sum of a quaternion norm form with a split form of dimension $r-4$. 
\end{enumerate}
\end{definition}

Put $\GL_2(\RR)^+=\{\alp\in\GL_2(\RR)\;|\;\det\alp>0\}$. 
For $\ell\in\ZZ$, $\alp=\Left a & b \\ c & d\Right\in\GL_2(\RR)^+$ and a function $f:\frkH\to\CC$, we define $f\|_{\ell/2}\alp:\frkH\to\CC$ by 
\[f\|_{\ell/2}\alp(\tau)=(\det\alp)^{\ell/4}(c\tau+d)^{-\ell/2}f(\alp\tau). \]
For $\gam\in J_S(\RR)$ and a function $\phi:\cald\to\CC$ we define $\phi\|_{\ell/2}\gam:\cald\to\CC$ similarly. 
We drop the subscript $\ell/2$ when there is no fear of confusion. 

We write $\cals(X_\bff)$ for the Schwartz-Bruhat space on $X_\bff=X\otimes_\QQ\AAf$. 
Given $l\in \cals(X_\bff)$, we define the theta function $\vth^S_l(\tau, w)$ by
\[\vth^S_l(\tau, w)=\sum_{\alp\in X}l(\alp)\bfe(\tau S[\alp]/2+S(\alp, w)). \]
Choose a complete representative $\Xi$ for $L^*/L$. We write $l_\mu$ for the characteristic function of the closure of $\mu+L$ in $X_\bff$ for $\mu\in\Xi$. 
To simplify notation, we put $\vth^S_\mu=\vth^S_{l_\mu}$. 

For the proof of the following lemma, see \cite[Lemma 4.9.1]{Mi}. 

\begin{lemma}\label{lem:11}
Put $J=\onon$. For each $\mu\in\Xi$, we have
\[\vth^S_\mu\|_{n/2}J=(\det S)^{-1/2}\bfe(-n/8)\sum_{\nu\in\Xi}\bfe(S(\mu,\nu))\vth^S_\nu. \]
\end{lemma}

For a rational number $N$, we denote by $\psi_N$ the primitive Dirichlet character corresponding to $\QQ(\sqrt{N})/\QQ$. 

\begin{proposition}\label{prop:11}
Let $\gam=\Left a & b \\ c & d\Right\in\Gam_0(D_S)$. 
\begin{enumerate}
\renewcommand\labelenumi{(\theenumi)}
\item If $n$ is odd, then 
\[j(\gam,\tau)^{-n}\bfe(-c(c\tau+d)^{-1}S[w]/2)\vth^S_0(\gam(\tau,w))=\biggl(\frac{4b_S}{d}\biggl)\vth^S_0(\tau,w). \]
For the definition of $j(\gam,\tau)$, see \S \ref{sec:2}.  
\item If $n$ is even, then 
\[\vth^S_0|_{n/2}\gam(\tau,w)=\chi_S(d)\vth^S_0(\tau,w). \]
\end{enumerate}
\end{proposition}
\begin{proof}
Note that $\psi_{2\det S}=\psi_{4b_S}$. 
Since $D_S S^{-1}$ is an even integral symmetric matrix, we can apply Corollary 4.9.5, Theorem 4.9.3 of \cite{Mi} and its remark to the theta constant $\vth^S_0(\tau,0)$, which proves to be a modular form of weight $n/2$ with respect to $\Gam_0(D_S)$. 

The proof is applicable to $\vth^S_0(\tau,w)$ itself. 
\end{proof}

A Jacobi form $\phi$ can be expressed as a sum
\begin{align}
\phi(\tau, w)&=\sum_{\mu\in\Xi}\phi_\mu(\tau)\vth^S_\mu(\tau, w), \label{tag:12}\\
\phi_\mu(\tau)&=\sum_ac_\phi(a, \mu)\bfe((a-S[\mu]/2)\tau) \label{tag:13}
\end{align}  
as a simple consequence of the invariance of $\phi$ with respect to $H_S(\ZZ)$. 

The following Lemma is a special case of \cite[Theorem 3.3]{Z}. 

\begin{lemma}\label{lem:12}
Let $\phi$ be a holomorphic function on $\cald$ which admits a expansion of the form (\ref{tag:12}), (\ref{tag:13}). Then $\phi\in J_{\kap,S}$ if and only if 
\[\phi_\mu\|_{\kap-n/2}J=(\det S)^{-1/2}\bfe(n/8)\sum_{\nu\in\Xi}\bfe(-S(\mu,\nu))\phi_\nu\]
for every $\mu\in\Xi$. Here we put $J=\onon$. 
\end{lemma}

We remark that the ``only if'' part is clear by Lemma \ref{lem:11}. 

\begin{definition}\label{def:12}
A Jacobi form $\phi\in J_{\kap,S}$ is an element of $J^M_{\kap,S}$ if it admits a Fourier expansion of the form
\[\phi(\tau,w)=\sum_{(a,\alp)\in \calt^0\cup\calt^+}c_\phi(D_{a,\alp})q^a\bfe(S(\alp,w))\] 
for some function $c_\phi:\NN\cup\{0\}\to \CC$.  
Put $J^{\cusp,M}_{\kap,S}=J^M_{\kap,S}\cap J^\cusp_{\kap,S}$. 
\end{definition}

\begin{remark}\label{rem:11}
Note that $\kap$ must be an even integer for there to be any non-zero $\phi\in J^M_{\kap,S}$. 
Therefore, whenever we consider the space $J^M_{\kap,S}$, we assume that $\kap$ is even. 
\end{remark}

The following lemma easily follows from (\ref{tag:13}). 

\begin{lemma}\label{lem:13}
Let $\phi\in J_{\kap,S}$. 
The following conditions are equivalent.
\begin{enumerate}
\renewcommand\labelenumi{(\theenumi)}
\item[(\roman{one})] $\phi\in J^M_{\kap,S}$. 
\item[(\roman{two})] $\phi_\mu=\phi_\nu$ if $S[\mu]/2-S[\nu]/2\in\ZZ$.  
\end{enumerate} 
\end{lemma}

It is important to note that we can recover $\phi$ from $\phi_0$ if $\phi\in J^M_{\kap,S}$. 
More precisely, we have 
\begin{lemma}\label{lem:14}
Let $\phi\in J^M_{\kap,S}$. 
Put 
\beq
h=(\det S)^{1/2}\bfe(-n/8)D_S^{-k/2-1/4}\phi_0\|J\mtrx{D_S}{}{}{1}. \label{tag:14}
\eeq
For an integer $\ell$, we put 
\[\bfa_S(\ell)=\sharp\{\mu\in\Xi\;|\;\ell\equiv -D_SS[\mu]/2\pmod{D_S}\}. \] 
Then 
\beq
\phi_\mu=D_S^{k/2-3/4}\bfa_S(D_{0,\mu})^{-1}\sum_{j=0}^{D_S-1}\bfe(jS[\mu]/2)h\|\mtrx{1}{j}{0}{D_S}. \label{tag:15}
\eeq
\end{lemma}

\begin{proof}
Observing that $\phi_\mu\|\Left 1 & 1 \\ 0 & 1\Right=\bfe(-S[\mu]/2)\phi_\mu$ and 
\[h=D_S^{-k/2-1/4}\sum_{\mu\in\Xi}\phi_\mu\|\mtrx{D_S}{}{}{1}, \]
we can conclude our lemma by an immediate computation. 
\end{proof}

For basic facts on the Hecke algebra acting on Jacobi forms we refer the reader to \cite{Su}. 
\begin{definition}\label{def:13}
For a Hecke eigenform $\phi\in J^\cusp_{\kap,S}$, the definition of its $L$-function $L(\phi,s)$ is given by 
\[L(\phi,s)=\prod_pP_{S,\lam_\phi,p}(p^{-s})^{-1}, \]
where $P_{S,\lam_\phi,p}(p^{-s})$ is the full denominator of the local $L$-factor introduced in \cite[(2.13)]{Su} (even though its factors can be canceled with those in the numerator). 
\end{definition}

For later use, we record the following proposition. 

\begin{proposition}\label{prop:12}
If $\phi\in J^\cusp_{\kap,S}$ is a Hecke eigenform and if $(a,\alp)\in\calt^+$ is such that $S^\sim=S_{a,\alp}$ is maximal, then 
\begin{multline*}
\sum_{m=1}^\infty c_\phi(am^2,\alp m)m^{-(s+\kap-1-n/2)}
\prod_p\frac{B_{S^\sim,p}(p^{-s-1/2})}{B_{S,p}(p^{-s})}\\
=c_\phi(a,\alp)L(\phi,s)\times
\begin{cases}
L\bigl(s+\thalf,\psi_{(-1)^{(n+1)/2}\Del_{a,\alp}}\bigl)^{-1} &\tx{if $2\nmid n$. }\\
\zet(2s)^{-1} &\tx{if $2|n$. }
\end{cases}
\end{multline*}
Here, we define $B_{S,p}$ by 
\[B_{S,p}(X)=\begin{cases}
1 &\tx{if $p\in\frkS_0$, or $p\in\frkS_1$, $2|n$ }\\
1+\eta_p(S)p^{1/2}X &\tx{if $p\in\frkS_1$, $2\nmid n$ }\\
1-pX^2 &\tx{if $p\in\frkS_2$, $2\nmid n$ }\\
(1-\xi_p(S)pX)(1-\xi_p(S)X) &\tx{if $p\in\frkS_2$, $2|n$ }
\end{cases}\]
and define $B_{S^\sim,p}$ similarly. 
\end{proposition}

This is a restatement of \cite[Proposition 4.2]{Su}. 

%%%%%%%%%%%%%%%%%%%%%%%%%%%%%%%%%%%%%%%%%%%%%%%%%%%%%%%%%%%%%%%%%%%%%%%%%

\section{\bf The modular forms of half-integral weight}\label{sec:2} 

We recall some basic facts about the modular forms of half-integral weight. 
We refer to \cite{UY, Ko, U} for detail. 
Fix a positive integer $k$. 
The set $\frkG$ consists of all pairs $(\gam,\phi(\tau))$, where $\gam=\Left a & b \\ c & d\Right\in\GL_2(\RR)^+$ and $\phi(\tau)$ is a holomorphic function on $\frkH$ satisfying 
\[|\phi(\tau)|=(\det\gam)^{-k/2-1/4}|c\tau+d|^{k+1/2}. \] 
We define the group law of $\frkG$ by
\[(\gam_1,\phi_1(\tau))\cdot(\gam_2,\phi_2(\tau))=(\gam_1\gam_2, \phi_1(\gam_2\tau)\phi_2(\tau)). \]  
For a function $g:\frkH\to\CC$ and for $\alp=(\gam,\phi(\tau))\in\frkG$, we put
\[g|\alp(\tau)=\phi(\tau)^{-1}g(\gam\tau). \] 

There exists an injective homomorphism $\Gam_0(4)\to\frkG$ given by 
$\gam\mapsto\gam^*=(\gam, j(\gam,\tau)^{2k+1})$, where 
\[j(\gam,\tau)=\Big(\frac{c}{d}\Big)\eps^{-1}_d(c\tau+d)^{1/2}\]
for $\gam=\Left a & b \\ c & d\Right\in\Gam_0(4)$. Here $\bigl(\frac{c}{d}\bigl)$ is the Kronecker symbol (see \cite{Mi}) and 
\[\eps_d=\begin{cases}
1 &\tx{if $d\equiv 1\pmod 4$. }\\
\iu &\tx{if $d\equiv 3\pmod 4$. }
\end{cases}\]

Fix a positive integer $N=2^eM$, where $M$ is an odd square-free integer and $e$ equals either $2$ or $3$. 
Let $\chi$ be an even Dirichlet character mod $N$ such that $\chi^2=1$. Put $\chi(\gam)=\chi(d)$ for $\gam=\Left a & b \\ c & d\Right\in\Gam_0(N)$. 

We call a holomorphic function $g$ on $\frkH$ a modular (resp. cusp) form of weight $k+1/2$ with respect to $\Gam_0(N)$ and $\chi$ if $g|\gam^*=\chi(\gam)g$ for every $\gam\in\Gam_0(N)$ and it is holomorphic (resp. vanishes) at all cusps. 

The space of modular (resp. cusp) forms of weight $k+1/2$ with respect to $\Gam_0(N)$ and $\chi$ is denoted by $M_{k+1/2}(N,\chi)$ (resp. $S_{k+1/2}(N,\chi)$). We write $M_{k+1/2}(N)=M_{k+1/2}(N,\chi)$ and $S_{k+1/2}(N)=S_{k+1/2}(N,\chi)$ if $\chi$ is the trivial character. 

Putting 
\[\frkD_k=\{m\in\NN\cup\{0\}\;|\;(-1)^km\equiv0,1\pmod 4\}, \] 
we define the Kohnen plus space $M^+_{k+1/2}(N)$ by
\[M^+_{k+1/2}(N)=\{g\in M_{k+1/2}(N)\;|\;c_g(m)=0\tx{ unless }m\in\frkD_k\}.\] 
Put 
\[S^+_{k+1/2}(N)=S_{k+1/2}(N)\cap M^+_{k+1/2}(N).\] 

We define the $\CC$-linear map $\wp_k$ on formal power series by 
\[\sum_mc(m)q^m|\wp_k=\sum_{m\in\frkD_k}c(m)q^m. \]
It is known that $\wp_k$ induces a $\CC$-linear isomorphism of $S_{k+1/2}(4M)$ onto $S_{k+1/2}^+(8M)$ \cite{UY}. 
We define $\til{\del}_a\in\frkG$ and an operator $U(a)$ on formal power series by 
\begin{gather*}
\til{\del}_a=\bigl(\mtrx{a}{}{}{1}, a^{-k/2-1/4}\bigl), \\
\sum_mc(m)q^m|U(a)=\sum_mc(am)q^m 
\end{gather*}
for a positive integer $a$. Put $U_k(a^2)=U(a^2)\wp_k$. 

We choose an element $\gam_Q\in\SL_2(\ZZ)$ such that
\[\gam_Q\equiv\begin{cases}
\mtrx{0}{-1}{1}{0} &\pmod{Q^2} \\
\ono_2 &\pmod{(Q^{-1}N)^2}  
\end{cases}\]
for each positive divisor $Q$ of $N$ such that $Q$ and $N/Q$ are coprime. 
When $Q$ is odd, we define operators $\til{W}(Q)$, $\til{Y}(Q)$ $\til{W}(2^eQ)$ and $\til{Y}(2^eQ)$ on $M_{k+1/2}(N)$ via 
\begin{align*}
\til{W}(Q)&=\gam_Q^*\til{\del}_Q, \\
\til{Y}(Q)&=Q^{-k/2+3/4}U(Q)\til{W}(Q), \\
\til{W}(2^eQ)&=\til{W}(Q^{-1}M)\tau(N), \\
\til{Y}(2^eQ)&=(2^eQ)^{-k/2+3/4}U(2^eQ)\til{W}(Q^{-1}M)\tau(N),  
\end{align*}
where we put
\[\tau(N)=\Big(\mtrx{0}{-1}{N}{0}, (N^{1/4}(-\iu \tau)^{1/2})^{2k+1}\Big)\in\frkG. \]

Put $\mu=(-1)^{[(k+1)/2]}$. 
\begin{proposition}[Kohnen]\label{prop:21}
We have  
\[M^+_{k+1/2}(4M)=\{h\in M_{k+1/2}(4M)\;|\;(2^{3/2}\mu)^{-1}h|\til{Y}(4)=h\}. \]
The following direct sum decomposition holds.   
\[M_{k+1/2}(4M)=M^+_{k+1/2}(4M)\oplus M^-_{k+1/2}(4M), \]
where
\[M^-_{k+1/2}(4M)=\{h\in M_{k+1/2}(4M)\;|\;(2^{1/2}\mu)^{-1}h|\til{Y}(4)=-h\}. \]
\end{proposition}

The Petersson inner products on $S_{k+1/2}(N)$ are defined by
\[\La g,h\Ra=[\SL_2(\ZZ):\Gam_0(N)]^{-1}\int_{\Gam_0(N)\bsl\frkH}g(\tau)\oln{h(\tau)}y^{k-3/2}dxdy, \]
where $\tau=x+\iu y$, for $g$, $h\in S_{k+1/2}(N)$. 

The space of newforms $S^{\new,+}_{k+1/2}(N)$ for $S^+_{k+1/2}(N)$ is the orthogonal complement of 
\[\sum_{p|4^{-1}N}\Big(S^+_{k+1/2}(p^{-1}N)+S^+_{k+1/2}(p^{-1}N)|U_k(p^2)\Big) \]
in $S_{k+1/2}^+(N)$ with respect to the Petersson inner product. 

We denote by $\til{T}(p^2)$ (resp. $T(p)$) the usual Hecke operator on the space of modular forms of half-integral (resp. integral) weight. 
The following result was given by Kohnen in \cite{Ko} provided $e=2$. 
The case that $e=3$ was obtained in \cite{UY}. 

\begin{proposition}\label{prop:22} 
\begin{enumerate}
\renewcommand\labelenumi{(\theenumi)}
\item We have
\[S^+_{k+1/2}(N)=\oplus_{a,d\geq 1,\; ad|4^{-1}N}S^{\new,+}_{k+1/2}(4a)|U_k(d^2). \]
\item The operators $\til{T}(p^2)$ and $U_k(q^2)$, where $(p, 4^{-1}N)=1$ and $q|4^{-1}N$, fix $S^{\new,+}_{k+1/2}(N)$. Moreover, $S^{\new,+}_{k+1/2}(N)$ has an orthogonal $\CC$-basis which consists of common Hecke eigenforms of these operators. 
\item There is a bijective correspondence, up to scalar multiple, between Hecke eigenforms in $S_{k+1/2}^{\new,+}(N)$ and those in $S^\new_{2k}(4^{-1}N)$ in the following way. If $g\in S_{k+1/2}^{\new,+}(N)$ is a Hecke eigenform, i.e.,    
\begin{align*}
g|\til{T}(p^2)&=\lam_pg, & g|U_k(q^2)&=\lam_qg
\end{align*} 
for every prime number $p\nmid 4^{-1}N$ and prime divisor $q$ of $4^{-1}N$,  
then there exists a primitive form $f\in S_{2k}(\Gam_0(4^{-1}N))$ 
such that 
\begin{align*}
f|T(p)&=\lam_pf, & f|U(q)&=\lam_qf
\end{align*} 
for every prime number $p\nmid 4^{-1}N$ and prime divisor $q$ of $4^{-1}N$. 
\end{enumerate}
\end{proposition}

For later convenience, we note the following: 

\begin{lemma}\label{lem:21}
Let $f$ be a primitive form in $S^\new_{2k}(4^{-1}N)$ and $g$ a corresponding Hecke eigenform in $S^{\new,+}_{k+1/2}(N)$. 
For each prime divisor $p$ of $4^{-1}N$, let $W(p)$ be the Atkin-Lehner involution on $S^\new_{2k}(4^{-1}N)$ and put 
\[\til{Y}\Pr(p)=\begin{cases}
\eps_p^{2k+1}p^{-1/2}\til{Y}(p) &\tx{if $2\neq p|4^{-1}N$. }\\
4^{-1}\mu \til{Y}(8) &\tx{if $2=p|4^{-1}N$. }
\end{cases}\]
For $\eps\in\{\pm 1\}$ the following conditions are equivalent:
\begin{enumerate}
\renewcommand\labelenumi{(\theenumi)} 
\item[(\roman{one})] $c_f(p)=\eps p^{k-1}$; 
\item[(\roman{two})] $f|W(p)=-\eps f$; 
\item[(\roman{thr})] $g|\til{Y}\Pr(p)=-\eps g$; 
\item[(\roman{fou})] $c_g(p^2m)=\eps p^{k-1}c_g(m)$ for all $m$; 
\item[(\roman{fiv})] $c_g(m)=0$ if $\Psi_p((-1)^km)=\eps$. 
\end{enumerate}
\end{lemma}

\begin{proof}
The proof is an easy paraphrase of \cite{UY}, and is omitted. 
\end{proof}

%%%%%%%%%%%%%%%%%%%%%%%%%%%%%%%%%%%%%%%%%%%%%%%%%%%%%%%%%%%%%%%%%%%%%%%%%

\section{\bf Statement of the main theorems}\label{sec:3}

We first define a bit more notation.  
The letter $k$ hereafter stands for a positive integer such that 
\[\kap=k+\bigl[\tfrac{n+1}{2}\bigl]\]
is an even integer. 
For $\eps\in\{\pm1\}$ and $e\in\ZZ$, we define a Laurent polynomials $l_{e, \eps}$ and $h_{e, \eps}$ by 
\begin{align*}
l_{e, \eps}(X)&=\begin{cases}
\dfrac{(\eps X)^{e+1}-X^{-e-1}}{\eps X-X^{-1}} &\tx{if $e\geq 0$, }\\
0 &\tx{if $e<0$, }
\end{cases}\\
h_{e, \eps}(X)&=\begin{cases}\eps X^e+X^{-e} &\tx{if $e> 0$. }\\
(1+\eps)/2 &\tx{if $e=0$. }\\
0 &\tx{if $e<0$. }
\end{cases}
\end{align*}
Let $a$ be a nonzero element of $\QQ_p$. Put $\Psi_p(a)=1$, $-1$, $0$ according as $\QQ_p(\sqrt{a})$ is $\QQ_p$, an unramified quadratic extension of $\QQ_p$ or a ramified quadratic extension of $\QQ_p$. 
If $p$ is an odd rational prime, then $\bigl(\frac{*}{p}\bigl)$ is the quadratic residue symbol modulo $p$. 
We put 
\[\biggl(\frac{a}{2}\biggl)=\begin{cases}
1 &\tx{if $a\equiv 1\pmod 8$. }\\
-1 &\tx{if $a\equiv 5\pmod 8$. }\\
0 &\tx{if $a\equiv 0,2,3,4,6,7\pmod 8$. }
\end{cases}\]
Let $\bigl(\frac{a}{p}\bigl)=0$ if $a\notin \ZZ_p$. 
Put
\[\frkf_p(a)=\biggl[\frac{\ord_pa+1-\del(p=2)}{2}\biggl]-1+\Psi_p((-1)^ka)^2. \]
We put $l_e=l_{e,1}$ and 
\begin{gather*}
\lam_{p,a}=l_{\frkf_p(a)}-\Psi_p((-1)^ka)p^{-1/2}l_{\frkf_p(a)-1}, \\
l_{p,S,a}=\begin{cases}
\lam_{p,a}
&\tx{if $p\in\frkS_0$, $2\nmid n$. }\\
\lam_{p,a}+\eta_p(S)p^{1/2}\lam_{p,p^{-2}a}
&\tx{if $p\in\frkS_1$, $2\nmid n$. }\\
\lam_{p,a}-\biggl(\dfrac{(-1)^kp^{-2}a}{p}\biggl)p^{1/2}\lam_{p,p^{-2}a}-p\lam_{p,p^{-4}a}
&\tx{if $p\in\frkS_2$, $2\nmid n$. }\\
l_{\ord_pa, \xi_p(S)}&\tx{if $p\in\frkS_0$, $2|n$. }\\
h_{\ord_pa, \eta_p(S/2)\Chi_{S,p}((-1)^kd_Sa)}&\tx{if $p\in\frkS_1$, $2|n$. }\\
l_{\ord_pa, \xi_p(S)}-\xi_p(S)pl_{\ord_pa-2, \xi_p(S)}&\tx{if $p\in\frkS_2$, $2|n$. }
\end{cases}
\end{gather*}
Set $\frkf_N=\prod_pp^{\frkf_p(N)}$ and $\frkd_N=N\frkf_N^{-2}$ for each positive integer $N$. 
\begin{theorem}\label{thm:31}
Suppose that $n$ is odd and $\kap=k+\frac{n+1}{2}$ is even. 
Let $b$ (resp. $d$) be a positive divisor of $b_S$ (resp. $d_S$). 
Let $f\in S_{2k}(\Gam_0(bd))$ be a primitive form, the $L$-function of which is given by 
\begin{multline*}
L(f,s)=\prod_{p|bd}(1-\alp_pp^{k-1/2-s})^{-1}\\
\times\prod_{p\nmid bd}(1-\alp_pp^{k-1/2-s})^{-1}(1-\alp_p^{-1}p^{k-1/2-s})^{-1}. 
\end{multline*}
Suppose that $f|W(p)=\eta_p(S)f$ for each prime factor $p$ of $b$. 
Let $g\in S^+_{k+1/2}(4bd)$ be a Hecke eigenform which corresponds to $f$ via the Shimura correspondence. 
For each positive integer $N$, we put 
\[c_\vPh(N)=2^{-\bfb_{bd}(N)}c_g(\frkd_{N})\frkf_{N}^{k-1/2}\prod_pl_{p,S,N}(\alp_p) \]
where $\bfb_{bd}(N)$ is the number of distinct prime divisors of $bd$ such that $\Psi_p((-1)^kN)\neq 0$.   
Then 
\[\vPh(\tau,w)=\sum_{(a,\alp)\in\calt^+}c_\vPh(\Del_{a,\alp})q^a\bfe(S(\alp,w))\]
is a Hecke eigenform in $J^{\cusp,M}_{\kap,S}$ whose $L$-function is given by 
\[L(\vPh,s)=\prod_p(1-\alp_pp^{-s})^{-1}(1-\alp_p^{-1}p^{-s})^{-1}. \] 
Moreover, the lifting $f\mapsto\vPh$ gives a bijective correspondence, up to scalar multiple, between Hecke eigenforms in 
\[\oplus_{b,d\geq 1,\;b|b_S,\; d|d_S}\{f\in S^\new_{2k}(bd)\;|\;f|W(p)=\eta_p(S)f\tx{ for each }p|b\}, \]
and those in $J^{\cusp,M}_{\kap,S}$. 
\end{theorem}

\begin{remark}\label{rem:31}
Theorem \ref{thm:31} is compatible with the result of \cite{SZ} when $n=1$, as one can see by an easy computation.  
\end{remark}

From now on we consider the case when $n$ is even. 
Let $d$ be a square-free integer such that $\frkd_S$ and $d$ are coprime. 
We write $\Prm_k(\frkd_Sd, \chi_S)$ for the set of all primitive forms in $S_k(\Gam_0(\frkd_Sd),\chi_S)$. 
To state our result, we need the subset $V_k(\frkd_Sd,\chi_S)(\subset\Prm_k(\frkd_Sd,\chi_S))$ of Definition \ref{def:62}. 
Let $\Prm^*_k(\frkd_Sd,\chi_S)$ be the complement of $V_k(\frkd_Sd,\chi_S)$ in $\Prm_k(\frkd_Sd,\chi_S)$. 
Note that Atkin-Lehner involutions on $S_k(\Gam_0(\frkd_Sd),\chi_S)$ acts on the set $\Prm^*_k(\frkd_Sd,\chi_S)$ (see \S \ref{sec:6}). 

\begin{theorem}\label{thm:32}
Suppose that $n$ and $\kap=k+\frac{n}{2}$ are even.  
Let $d$ be a positive divisor of $d_S$ and $f$ an element of $\Prm^*_k(\frkd_Sd,\chi_S)$, the $L$-function of which is given by
\begin{multline*}
L(f,s)=\prod_{p|d}(1-\xi_p(S)\alp_pp^{(k-1)/2-s})^{-1}\prod_{p|\frkd_S}(1-\alp_pp^{(k-1)/2-s})^{-1}\\
\times\prod_{p\nmid \frkd_Sd}(1-\xi_p(S)\alp_pp^{(k-1)/2-s})^{-1}(1-\alp_p^{-1}p^{(k-1)/2-s})^{-1}. 
\end{multline*}
For each positive integer $N$, we put 
\[c_\vPh(N)=N^{(k-1)/2}\prod_pl_{p,S,N}(\alp_p)\]
We define the function $\vPh$ on $\cald$ by
\[\vPh(\tau,w)=\sum_{(a,\alp)\in\calt^+}c_\vPh(D_{a,\alp})q^a\bfe(S(\alp,w)). \]
Then $\vPh$ is a Hecke eigenform in $J^{\cusp,M}_{\kap,S}$ whose $L$-function is given by
\[L(\vPh,s)=\prod_p(1-\alp_p^2p^{-s})^{-1}(1-\xi_p(S)p^{-s})^{-1}(1-\alp_p^{-2}p^{-s})^{-1}. \]
Moreover, the lifting $f\mapsto\vPh$ defines a bijection from the orbits of Atkin-Lehner involutions in 
\[\bigcup_{d\geq 1, \; d|d_S}\Prm^*_k(\frkd_Sd,\chi_S) \]
onto Hecke eigenforms in $J^{\cusp,M}_{\kap,S}$. 
\end{theorem}

%%%%%%%%%%%%%%%%%%%%%%%%%%%%%%%%%%%%%%%%%%%%%%%%%%%%%%%%%%%%%%%%%%%%%%%%%

\section{\bf The space $M^S_{k+1/2}(\vDe_S)$}\label{sec:4} 

Suppose that $n$ is odd and continue with this assumption until \S \ref{sec:5}.  
We first put $\vDe_S=4b_Sd_S$. 
For each prime factor $q$ of $4^{-1}\vDe_S$, we define the trace operator $\Tr^S_q: M^+_{k+1/2}(\vDe_S)\to M^+_{k+1/2}(q^{-1}\vDe_S)$ by putting
\[\Tr^S_q(g)=(q+1)^{-1}\sum_{\gam\in\Gam_0(\vDe_S)\bsl\Gam_0(q^{-1}\vDe_S)}g|\gam^* \]
for $g\in M^+_{k+1/2}(\vDe_S)$ if $q$ is odd, and by putting
\[\Tr^S_2(g)=\frac{1}{2}\biggl(\sum_{\gam\in\Gam_0(\vDe_S)\bsl\Gam_0(2^{-1}\vDe_S)}g|\gam^*\biggl)|\mathrm{pr}\]if $q=2$. 
Here, let
\[\mathrm{pr}=(2^{-1/2}\mu \til{Y}(4)+1)/3. \]

\begin{remark}\label{rem:41}
\begin{enumerate}
\renewcommand\labelenumi{(\theenumi)} 
\item Proposition \ref{prop:21} states that the operator $\mathrm{pr}$ is the orthogonal projection of $M_{k+1/2}(2^{-1}\vDe_S)$ onto $M^+_{k+1/2}(2^{-1}\vDe_S)$. 
\item If $q$ is odd, then $\Tr^S_q=(q+1)^{-1}(1+q^{-k/2+3/4}\til{W}(q)U(q))$. Indeed, 
\[(q+1)\Tr^S_q(g)=g+\sum_{i=0}^{q-1}g|\gam_q^*\mtrx{1}{i}{0}{1}^*=g+q^{-k/2+3/4}g|\til{W}(q)U(q), \]
where we use the fact hat $\{\ono_2\}\cup\{\gam_q\Left 1 & i \\ 0 & 1\Right\}_{i=0}^{q-1}$ is a complete representative for $\Gam_0(\vDe_S)\bsl\Gam_0(q^{-1}\vDe_S)$.  
\end{enumerate}
\end{remark}

We consider the following conditions: 
\begin{enumerate}
\renewcommand\labelenumi{(\theenumi)}
\item[(\roman{one})] $c_g(m)=0$ if $\biggl(\dfrac{(-1)^km}{q}\biggl)=-\eta_q(S)$;   
\item[(\roman{two})] $\Tr^S_q(g)=0$ 
\end{enumerate} 
for each prime factor $q$ of $4^{-1}\vDe_S$. 

\begin{lemma}\label{lem:41}
If $q$ is an odd prime divisor of $4^{-1}\vDe_S$, then the following assertions hold. 
\begin{enumerate}
\renewcommand\labelenumi{(\theenumi)}
\item $\eps_q^{2k+1}q^{-1/2}\til{Y}(q)$ is an involution on $M^+_{k+1/2}(\vDe_S)$. 
\item The conditions (\roman{one}) holds if and only if $\eps_q^{2k+1}q^{-1/2}g|\til{Y}(q)=\eta_q(S)g$. 
\item The conditions (\roman{two}) holds if and only if $g|\til{W}(q)\til{Y}(q)=-g|\til{W}(q)$.  
\end{enumerate}
\end{lemma}
\begin{proof}
See \cite[Proposition 1.29]{U} for our assertions (1), (2). 
Our assertion (3) follows from Remark \ref{rem:41} (2). 
\end{proof}

\begin{lemma}\label{lem:42}
If $\vDe_S$ is divisible by $8$, then the following assertions hold. 
\begin{enumerate}
\renewcommand\labelenumi{(\theenumi)}
\item The operator $(4\mu)^{-1}\til{Y}(8)$ is an involution on $M_{k+1/2}^+(\vDe_S)$. 
\item The conditions (\roman{one}) holds if and only if $(4\mu)^{-1}g|\til{Y}(8)=\eta_2(S)g$.  
\item The conditions (\roman{two}) holds if and only if $(2^{1/2}\mu)^{-1}g|\wp_k^{-1}\til{Y}(4)=-g$. 
\end{enumerate}
\end{lemma}

\begin{proof}
See \cite{UY} for our assertion (1), (2). 
Our assertion (3) follows from Remark \ref{rem:41} (1) and the fact that $\wp_k^{-1}$ and the $\CC$-linear map
\[g\mapsto2^{-1}\biggl(\sum_{\gam\in\Gam_0(8M)\bsl\Gam_0(4M)}g|\gam^*\biggl)(3-2^{-1/2}\mu\til{Y}(4))\] 
agree on $M^+_{k+1/2}(8M)$ (see \cite{UY}). 
\end{proof}

We now introduce the intermediate space $M^S_{k+1/2}(\vDe_S)$. 

\begin{definition}\label{def:41}
The space $M^S_{k+1/2}(\vDe_S)$ consists of all functions $g\in M^+_{k+1/2}(\vDe_S)$ with the following properties: 
\begin{enumerate}
\renewcommand\labelenumi{(\theenumi)}
\item[(A)] $g$ satisfies (\roman{one}) for every prime $q\in\frkS_1$;
\item[(B)] $g$ satisfies (\roman{two}) for every prime $q\in\frkS_2$. 
\end{enumerate}
We put $S^S_{k+1/2}(\vDe_S)=S_{k+1/2}(\vDe_S)\cap M^S_{k+1/2}(\vDe_S)$. 
\end{definition}

Let $b$ and $d$ be positive divisors of $b_S$ and $d_S$ respectively, and $p$ and $\ell$ prime divisors of $b^{-1}b_S$ and $d^{-1}d_S$ respectively. 
We define the operator $P(p): M^+_{k+1/2}(4bd)\to M^+_{k+1/2}(4pbd)$ by setting 
\[g|P(p)=g+\eta_p(S)p^{1-k}(g|\til{T}(p^2)-g|U_k(p^2)), \]
and define the operator $Q(\ell): M^+_{k+1/2}(4bd)\to M^+_{k+1/2}(4\ell bd)$ by setting
\[g|Q(\ell)=(\ell+1)g|U_k(\ell^2)-\ell g|\til{T}(\ell^2) \]
for $g\in M^+_{k+1/2}(4bd)$. 
\begin{remark}\label{rem:42}
Let us note that
\begin{align*}
g|P(p)(\tau)
&=\sum_m\biggl(1+\eta_p(S)\biggl(\frac{(-1)^km}{p}\biggl)\biggl)(c_g(m)+\eta_p(S)p^kc_g(p^{-2}m))q^m\\
g|Q(\ell)(\tau)
&=\sum_{m\in\frkD_k}\biggl(c_g(\ell^2m)-\ell^k\biggl(\frac{(-1)^km}{\ell}\biggl)c_g(m)-\ell^{2k}c_g(\ell^{-2}m)\biggl)q^m. 
\end{align*}
\end{remark}

We define $g^*\in M^+_{k+1/2}(\vDe_S)$ by 
\[g^*=g|\prod_{p|b^{-1}b_S}P(p)\prod_{p|d^{-1}d_S}Q(p) \]
for $g\in M^+_{k+1/2}(4bd)$. 

The space $\frkS^S_{k+1/2}(4bd)$ consists of all functions $g\in S^{\new,+}_{k+1/2}(4bd)$ which satisfy the condition (\roman{one}) for every prime divisor $q$ of $b$. 

\begin{proposition}\label{prop:41}
The mapping $g\mapsto g^*$ induces a $\CC$-linear isomorphism $i_2:\oplus_{b,d\geq 1,\; b|b_S, \;d|d_S}\frkS^S_{k+1/2}(4bd)\simeq S^S_{k+1/2}(\vDe_S)$. 
\end{proposition}

\begin{proof}
Let $g\in\frkS^S_{k+1/2}(4bd)$. 
Our first task is to prove $g^*\in S^S_{k+1/2}(\vDe_S)$. 
In view of the definition of $\frkS^S_{k+1/2}(4bd)$ and Remark \ref{rem:42}, we conclude that $g^*$ satisfies (A) and $\Tr^S_p(g^*)=0$ for every prime divisor $p$ of $d$. 

Suppose that $d^{-1}d_S$ is even. 
Since $S_{k+1/2}^+(4bd)$ is the $2^{3/2}\mu$-eigenspace of $\til{Y}(4)$ by Proposition \ref{prop:21},  
\begin{align*}
g|Q(2)\wp_k^{-1}\til{Y}(4)
&=3g|U_k(4)\wp_k^{-1}\til{Y}(4)-2g|\til{T}(4)\wp_k^{-1}\til{Y}(4)\\
&=3g|\wp_k^{-1}U(4)\til{Y}(4)-2g|\til{T}(4)\til{Y}(4)\\
&=3g|U(4)\til{Y}(4)-2^{5/2}\mu g|\til{T}(4). 
\end{align*}
Since
\[\til{T}(4)=\tfrac{3}{2}\cdot U(4)\mathrm{pr}=2^{-3/2}\mu U(4)\til{Y}(4)+\thalf U(4)\]
by the definition of $\til{T}(4)$, we have
\[g|Q(2)\wp_k^{-1}\til{Y}(4)=-2^{1/2}\mu(3g|U(4)-2g|\til{T}(4))=-2^{1/2}\mu g|Q(2)\wp_k^{-1}. \]

If $p$ is an odd prime divisor $d^{-1}d_S$, then we have
\[g|U(p^2)\til{W}(p)=g|\til{T}(p^2)\til{\del}_p-p^{k/2-3/4}g|U(p), \]
where we use relations   
\begin{align*}
g|\til{W}(p)^2&=\eps_p^{-2k-1}g, & g|\til{T}(p^2)-g|U(p^2)&=\eps_p^{2k+1}p^{k-3/2}g|\til{Y}(p) 
\end{align*} 
(see \cite[Proposition 1.18, (3.12)]{U}). 
Observe that
\begin{align}
g|Q(p)\til{W}(p)\label{tag:41}
&=(p+1)g|U(p^2)\til{W}(p)-pg|\til{T}(p^2)\til{\del}_p\\
&=g|\til{T}(p^2)\til{\del}_p-(p+1)p^{k/2-3/4}g|U(p). \notag  
\end{align}
Now $g|Q(p)\til{W}(p)\til{Y}(p)$ is equal to 
\begin{multline*}
p^{-k/2+3/4}g|\til{T}(p^2)\til{\del}_pU(p)\til{W}(p)-(p+1)g|U(p^2)\til{W}(p)\\
=pg|\til{T}(p^2)\til{\del}(p)-(p+1)(g|\til{T}(p^2)\til{\del}_p-p^{k/2-3/4}g|U(p))=-g|Q(p)\til{W}(p). 
\end{multline*}

By virtue of Lemma \ref{lem:41} (3) and \ref{lem:42} (3), $g^*$ satisfies (B) and hence $g^*\in S^S_{k+1/2}(\vDe_S)$ as expected. 

If $g|P(p)=0$, then $g=0$ since Proposition \ref{prop:22} (1) implies that 
\begin{align*}
g+\eta_p(S)p^{1-k}g|\til{T}(p^2)&=0, &g|U_k(p^2)&=0. 
\end{align*}
In a similar fashion we can show $g=0$ if $g|Q(p)=0$. 
Consequently, $i_2$ is injective. 

For a prime divisor $p$ of $4^{-1}\vDe_S$ and for $g\in S^S_{k+1/2}(\vDe_S)$, Proposition \ref{prop:22} (1) gives 
\begin{align*}
g_1&\in \oplus_{p|c,\; ac|4^{-1}\vDe_S}S^{\new,+}_{k+1/2}(4c)|U_k(a^2), & 
g_2,\; g_3&\in S^+_{k+1/2}(p^{-1}\vDe_S)
\end{align*} 
such that 
\[g=g_1+g_2+g_3|U_k(p^2). \] 
Put $h_1=g_2+g_3|U_k(p^2)$. 
Since $\til{Y}\Pr(p)$ fixes the space of newforms, Lemma \ref{lem:41} and \ref{lem:42} imply that $g_1$, $h_1\in S^S_{k+1/2}(\vDe_S)$ (see Lemma \ref{lem:21} for the definition of $\til{Y}\Pr(p)$). 

Assume that $p\in\frkS_1$. 
Put $h_2=h_1+\eta_p(S)p^{k-1}g_3|P(p)$. 
Since 
\[h_2=g_2+\eta_p(S)p^{k-1}g_3+g_3|\til{T}(p^2)\]
is an element of $S^+_{k+1/2}(p^{-1}\vDe_S)$ and $h_2|\til{Y}\Pr(p)=\eta_p(S)h_2$, we have $h_2=0$ and $h=g_1-\eta_p(S)p^{k-1}g_3|P(p)$ (see \cite{UY}). 

Assume that $p\in\frkS_2$. 
Put $h_3=h_1-(p+1)^{-1}g_3|Q(p)$. 
Since 
\[h_3=g_2+(p+1)^{-1}pg_3|\til{T}(p^2)\in S^+_{k+1/2}(p^{-1}\vDe_S), \]
we have $h_3=\Tr^S_p(h_3)=0$. 
Thus $h=g_1+(p+1)^{-1}g_3|Q(p)$. 
We can establish the surjectivity of $i_2$ by induction. 
\end{proof}

%%%%%%%%%%%%%%%%%%%%%%%%%%%%%%%%%%%%%%%%%%%%%%%%%%%%%%%%%%%%%%%%%%%%%%%%%

\section{\bf Proof of Theorem \ref{thm:31}}\label{sec:5}

We do not require $n$ to be odd in the following lemma. 

\begin{lemma}\label{lem:51}
Let $(U, Q)$ be a non-degenerate quadratic space over $\QQ$ with $\dim U\geq 4$ for which $U_v=U\otimes_\QQ\QQ_v$ is isotropic for some place $v$, $L$ a lattice on $U$ and $L_p$ its closure in $U_p$. 
If a non-zero rational number $a\in Q[U]$ satisfy $a\in Q[L_p]$ for every prime $p$, then there exists an element $z\in U$ satisfying the following conditions: 
\begin{enumerate}
\renewcommand\labelenumi{(\theenumi)}
\item[(\roman{one})] $z\in L_p$ for every prime $p\neq v$; 
\item[(\roman{two})] $Q[z]=a$. 
\end{enumerate}
\end{lemma}

\begin{proof}
This is a special case of \cite[Lemma 6.2.3]{Ki}. 
\end{proof}

In the rest of this section we suppose that $n$ is odd and $k$ is a positive integer such that $\kap=k+\tfrac{n+1}{2}$ is even. 
For each rational prime $p$ and $\ell\in\ZZ_p$, we put
\[\bfa_{S,p}(\ell)=\begin{cases}
1 &\tx{if $2\neq p\in\frkS_0$. }\\
1+\eta_p(S)\biggl(\dfrac{(-1)^k\Lam_S\ell}{p}\biggl) &\tx{if $2\neq p\in\frkS_1$. }\\
p\del(p\nmid \ell)+1 &\tx{if $2\neq p\in\frkS_2$. }\\
\del(\Lam_S\ell\in\frkD_k) &\tx{if $p=2\in\frkS_0$. }\\
\del(\Lam_S\ell\in\frkD_k)\biggl(1+\eta_2(S)\biggl(\dfrac{(-1)^k\Lam_S\ell}{2}\biggl)\biggl) &\tx{if $p=2\in\frkS_1$. }\\
\del(\Lam_S\ell\in\frkD_k)+3\del(\Lam_S\ell\notin\frkD_k) &\tx{if $p=2\in\frkS_2$. }
\end{cases}\]
 
\begin{lemma}\label{lem:52}
Let $\ell\in\NN$. The following conditions are equivalent: 
\begin{enumerate}
\renewcommand\labelenumi{(\theenumi)}
\item[(\roman{one})] all primes $p$ satisfy $\bfa_{S,p}(\ell)\neq 0$; 
\item[(\roman{two})] there exists $(a,\alp)\in\calt^+$ such that $\ell=D_{a,\alp}$. 
\end{enumerate}
Moreover, $\bfa_S(\ell)=\prod_p\bfa_{S,p}(\ell)$, where $\bfa_S(\ell)$ is defined in Lemma \ref{lem:14}.
\end{lemma}

\begin{proof}
We define the quadratic form $Q$ on $U=\QQ e\oplus X\oplus \QQ f$ by 
\[Q[ae+\alp+bf]=D_S(ab-S[\alp]/2). \]
Put $L_1=\ZZ e\oplus L\oplus\ZZ f$ and $L_1^*=\ZZ e\oplus L^*\oplus\ZZ f$. 
To prove the first part, it suffices to show 
\[Q[L_1^*]=\{\ell\in\ZZ\;|\;\bfa_{S,p}(\ell)\neq 0\tx{ for all }p\}. \]

We first note that the case $n=1$ is immediate since 
\begin{align*}
S/2&=b_S, & D_S&=4b_S, & L^*&=(2b_S)^{-1}L
\end{align*}
in this case. 
We hereafter assume that $n\geq 3$. 
Applying Lemma \ref{lem:51} with $v=\infty$, we have only to show 
\beq
\bfa_{S,p}(\ell)=\sharp\{\alp\in L^*_{1,p}/L_{1,p}\;|\;Q[\alp]\equiv\ell\pmod{D_S}\}\label{tag:51}
\eeq
for every prime $p$ and $\ell\in\ZZ_p$.

Fix a prime $p$ and put $\frka=D_S\ZZ_p$. 
Since $L_{1,p}$ is $\frka$-maximal, Lemma 6.5 of \cite{Sh2} gives a set of elements $\{e_i,f_i\}$ and an anisotropic kernel $Z_p$ of $(U_p,Q)$ such that  
\begin{gather*}
L_{1,p}=\sum_i(\ZZ_p f_i+\frka e_i)+M_p, \quad\quad
M_p=\{\alp\in Z_p\;|\; Q[\alp]\in\frka\}, \\
Q(e_i,e_i)=Q(f_j,f_j)=0, \quad 2Q(e_i,f_j)=\del(i=j)\; \tx{ for every $i$ and $j$}.   
\end{gather*}
Then we can observe that
\[L^*_{1,p}=\sum_i(\ZZ_p f_i+\frka e_i)+M^*_p, \]
where 
\[M^*_p=\{\alp\in Z_p\;|\;2Q(\alp,\bet)\in\frka \tx{ for every }\bet\in M_p\}. \]

Noting that  
\[(-1)^{(n+3)/2}\det Q=(-1)^{k+1}\Lam_S\in\QQ_p^\times/\QQ_p^{\times2} \]
(see (\ref{tag:11})), we have exactly two possibilities for the isometry class $Q$ distinguished by the sign $\eta_p(S)$. 
\begin{enumerate}
\renewcommand\labelenumi{(\theenumi)}
\item[(1)] Suppose that $\eta_p(S)=1$. 
Since $\dim Z_p=1$, the quadratic space $Z_p$ is isometric to $(\QQ_p, (-1)^k\Lam_S)$ by the observation above. We of course have $p\notin\frkS_2$ and conclude that 
\begin{align*}
M_p&=2b_S\ZZ_p, & 
M^*_p&=\ZZ_p. 
\end{align*}
\item[(2)] Next suppose that $\eta_p(S)=-1$. 
Let $H$ be a quaternion division algebra over $\QQ_p$, $\nu$ the reduced norm on $H$ and $B$ the set of pure quaternions in $H$. 
Since $\dim Z_p=3$, the quadratic space $Z_p$ is isometric to $(B, (-1)^{k+1}\Lam_S\cdot\nu)$. 
Letting $\calo$ be the maximal compact subring of $H$ and $\frkP$ the maximal ideal of $\calo$, we have   
\begin{align*}
M_p&=\begin{cases}
2\frkP\cap B &\tx{if $p|b_S$, }\\
2\calo\cap B &\tx{if $p\nmid b_S$, }
\end{cases} & 
M_p^*&=\begin{cases}
\calo\cap B &\tx{if $p|b_S$. }\\
\frkP^{-1}\cap B &\tx{if $p\nmid b_S$. }
\end{cases} 
\end{align*}
Note that $s_p(S)\geq 1$ in this case. 
\end{enumerate} 

We obtain (\ref{tag:51}) in both cases as expected. 
\end{proof}

Our immediate goal is to define a canonical $\CC$-linear map $i_1:J_{\kap,S}\to M^S_{k+1/2}(\vDe_S)$ whose restriction induces an isomorphism of $J^M_{\kap,S}$ onto $M^S_{k+1/2}(\vDe_S)$. 
To do this, we need some representation theoretic background for theta functions. 

Let $W=\QQ^2$ (row vectors) with the symplectic form defined via 
\begin{align*}
\La x, y\Ra&=xJ\trs y, & J&=\mtrx{0}{-1}{1}{0}. 
\end{align*} 
Let $\WW=X\otimes_\QQ W$ with a symplectic form $\ll\;,\;\gg=S(\;,\;)\otimes \La\;,\;\Ra$, and $\wtl{\Sp(\WW)}$ the metaplectic covering of $\Sp(\WW)(\AA)$. 
For any subgroup $H$ of $\Sp(\WW)$, we denote by $\wtl{H}$ the pull-back of $H$ in $\wtl{\Sp(\WW)}$. 

Let $v$ be a place of $\QQ$. 
As is well-known, $\wtl{\SL_2(\QQ_v)}=\wtl{\Sp(W)(\QQ_v)}$ is set-theoretically the product $\SL_2(\QQ_v)\times\{\pm1\}$ and the group law is given by $(g_1,\zet_1)(g_2,\zet_2)=(g_1g_2, c_v(g_1,g_2)\zet_1\zet_2)$, 
where $c_v(g_1,g_2)$ is the Kubota $2$-cocycle on $\SL_2(\QQ_v)$. 

The metaplectic cover $\wtl{\SL_2(\QQ_p)}\to\SL_2(\QQ_p)$ splits over the subgroup 
\[\Gam_1(4;\ZZ_p)=\Big\{\mtrx{a}{b}{c}{d}\in\SL_2(\ZZ_p)\;|\;c\equiv 0,\; d\equiv 1\pmod{4\ZZ_p}\Big\}. \] 
Of course, $\Gam_1(4;\ZZ_p)=\SL_2(\ZZ_p)$ if $p$ is odd. 
The splitting $\gam\mapsto(\gam,\sig_p(\gam))$ is given by
\[\sig_p\Big(\mtrx{a}{b}{c}{d}\Big)
=\begin{cases}
(c,d)_p&\tx{if $cd\neq0$ and $\ord_pc$ is odd. }\\
1 &\tx{otherwise. }
\end{cases}\]
The metaplectic cover $\wtl{\SL_2(\RR)}\to\SL_2(\RR)$ splits over $\Gam_1(4)$ and the splitting $\gam\mapsto(\gam,\sig_\infty(\gam))$ is given by 
\[\sig_\infty\Big(\mtrx{a}{b}{c}{d}\Big)=\Big(\frac{c}{d}\Big). \]
We identify $\Gam_1(4;\ZZ_p)$ and $\Gam_1(4)$ with the images of the splittings if there is no fear of confusion. 

We can define the automorphy factor $\jmath(\til\gam, \tau)$ on $\wtl{\SL_2(\RR)}\times\frkH$ by 
\[\jmath\Big(\Big(\mtrx{a}{b}{c}{d},\zet\Big),\tau\Big)
=\begin{cases}
\zet(c\tau+d)^{1/2} &\tx{if $c\neq 0$. }\\
\zet \sqrt{d} &\tx{if $c=0$, $d>0$. }\\
-\zet\iu\sqrt{|d|} &\tx{if $c=0$, $d<0$. }
\end{cases}\]
The map $\til\gam=(\gam,\zet)\mapsto(\gam,\jmath(\til\gam,\tau))$ is an isomorphism of $\wtl{\SL_2(\RR)}$ onto the commutator subgroup of $\frkG$, which on $\Gam_1(4)$ coincides with the homomorphism $\gam\mapsto\gam^*$.  

The group $\wtl{\SL_2(\QQ_v)}$ acts on $H_S(\QQ_v)$ through $\SL_2(\QQ_v)$. 
Hence we obtain the group 
\[\wtl{J_S(\QQ_v)}=\wtl{\SL_2(\QQ_v)}H_S(\QQ_v). \] 
Then $\wtl{J_S(\RR)}$ acts on $\cald$ through $J_S(\RR)$, and for each integer $\ell$, the automorphy factor $\jmath_\ell$ on $\wtl{J_S(\RR)}\times\cald$ is defined by 
\begin{multline*}
\jmath_\ell([\xi,\eta,\zet]\til\gam,(\tau,w))\\
=\jmath(\til\gam,\tau)^\ell
\bfe\biggl(-\zet+\frac{cS[w]}{2(c\tau+d)}-\frac{S(\xi,w)}{c\tau+d}-\frac{S[\xi]\cdot\gam\tau}{2}\biggl),  
\end{multline*}
where $\gam=\Left * & * \\ c & d\Right$ is the image of $\til\gam$ in $\SL_2(\RR)$. 

The local Weil representation of $\wtl{\Sp(\WW)(\QQ_v)}$ associated to $\bfe_v$ yields a representation $\ome_v$ of $\wtl{\SL_2(\QQ_v)}$. 
This representation can be realized on the Schwartz-Bruhat space $\cals(X_v)$ on $X_v$ with 
\begin{gather*}
\ome_v\Big(\mtrx{a}{b}{0}{a^{-1}},\zet\Big)l(\xi)=\zet\gam_{S,v}(1)\gam_{S,p}(a)^{-1}|a|_v^{n/2}\bfe_v(abS[\xi]/2)l(\xi a), \\
\ome_v(J,\zet)l(\xi)=\zet\gam_{S,v}(1)^{-1}\int_{X_v}l(\eta)\bfe_v(-S(\eta, \xi))d\eta. 
\end{gather*}
Here let $|\;|_v$ be the module of $\QQ_v^\times$ and $\gam_{S,v}$ the Weil constant with respect to $S$ and $\bfe_v$. 
The measure $d\eta$ is the self-dual Haar measure on $X_v$ with respect to the pairing $(\xi,\eta)\mapsto\bfe_v(S(\xi,\eta))$. 

The representation $\ome_v$ extends to the representation of $\wtl{J_S(\QQ_v)}$ by 
\beq
\ome_v([x,y,z])l(\xi)=\bfe_v(z+S(y,\xi))l(\xi+x).  \label{tag:52}
\eeq 
By the Stone-von Neumann theorem, (\ref{tag:52}) is a unique irreducible admissible representation of $H_S(\QQ_v)$ on which $[0,0,z]$ acts by $\bfe_v(z)$. 

We can identify $\wtl{\SL_2(\AA)}$ with the restricted direct product of $\wtl{\SL_2(\QQ_v)}$ with respect to $\{\Gam_1(4;\ZZ_p)\}$ divided by $\{(t_v)\in \oplus_v\{\pm1\} \;|\;\prod_vt_v=1\}$. 
Put $\wtl{J_S(\AA)}=\wtl{\SL_2(\AA)}H_S(\AA)$, $\wtl{J_S(\AAf)}=\wtl{\SL_2(\AAf)}H_S(\AAf)$ and $X_\AA=X\otimes_\QQ\AA$.  
The global Weil representation $\ome$ (resp. $\ome_\bff$) of $\wtl{J_S(\AA)}$ (resp. $\wtl{J_S(\AAf)}$) on $\cals(X_\AA)$ (resp. $\cals(X_\bff)$) is the restricted tensor product of $\ome_v$. 

For each $l\in\cals(X_\bff)$, we define $l\Pr\in\cals(X_\AA)$ by 
\[l\Pr(x,\xi)=e^{-\pi S[x]}l(\xi)\] 
for $x\in X_\infty$ and $\xi\in X_\bff$. 
The theta function $\vTh^S(\sig;l)$ is defined  by 
\begin{align*}
\vTh^S(\sig;l)&=\sum_{\alp\in X}\ome(\sig)l\Pr & 
\tx{for }\sig&\in J_S(\AA). 
\end{align*}
We now have an easy relation 
\begin{align*}
\vTh^S(g; l)&=\jmath_n(g, (\iu,0))^{-1}\vth^S_l(g(\iu, 0)), & g&\in\wtl{J_S(\RR)}, 
\end{align*}
where $\vth^S_l$ is the theta function defined in \S \ref{sec:1}. 

There is a unique splitting $\iot:\SL_2(\QQ)\to\wtl{\SL_2(\AA)}$, the image of which we identify with $\SL_2(\QQ)$, and $\vTh^S(g;l)$ is left invariant by $\SL_2(\QQ)$ by Weil's fundamental result. 
Let $\gam\in\SL_2(\QQ)$ and choose $\gam_\infty\in\wtl{\SL_2(\RR)}$ and $\gam_\bff\in\wtl{\SL_2(\AAf)}$ such that $\gam=\gam_\infty\gam_\bff$. 
Observe that 
\[\jmath_n(\gam_\infty,(\tau,w))^{-1}\vth^S_l(\gam(\tau,w))=\vth^S_{\ome_\bff(\gam_\bff^{-1})l}(\tau,w). \]
If $\gam\in\Gam_1(4)$, then we can take 
\begin{align*}
\gam_\infty&=(\gam,\sig_\infty(\gam)), & \gam_\bff&=\prod_p(\gam,\sig_p(\gam)), 
\end{align*} 
i.e., the splitting $\iot$ coincides with the map $\gam\mapsto\prod_v(\gam,\sig_v(\gam))$ on $\Gam_1(4)$. 
Note that there is no canonical choice of $\gam_\infty$ and $\gam_\bff$ if $\gam\notin\Gam_1(4)$. 

Suppose that $\gam\in\SL_2(\ZZ)$ and choose a matrix $u(\gam)_{\mu\nu}$ such that 
\[\ome_\bff(\gam_\bff^{-1})l_\nu=\sum_{\mu\in\Xi}u(\gam)_{\mu\nu}l_\mu\]
(see \S \ref{sec:1} for the definition of $l_\mu$). 
It follows that 
\[\vth^S_\nu\|_{n/2}\gam=t\sum_{\mu\in\Xi}u(\gam)_{\mu\nu}\vth^S_\mu\]
for some $t\in\CC$ with $|t|=1$. 
Since the natural inner product on $\cals(X_\AA)$ is $\wtl{\SL_2(\AA)}$-invariant, $u(\gam)_{\mu\nu}$ is a unitary matrix. 
Since $\{\vth^S_\mu\;|\;\mu\in\Xi\}$ is linearly independent, we have    
\beq
\phi_\nu\|_{k+1/2}\gam=t^{-1}\sum_{\mu\in\Xi}\oln{u(\gam)_{\mu\nu}}\phi_\mu \label{tag:53}
\eeq
for each $\phi\in J_{\kap,S}$. 
If $\gam\in\Gam_1(4)$, then more strongly 
\[\phi_\nu|\gam^*=\sum_{\nu\in\Xi}\oln{u(\gam)_{\mu\nu}}\phi_\mu. \]

Putting $\Xi(p)=L_p^*/L_p$, we can identify $\Xi$ with the direct sum $\oplus_p\Xi(p)$ in the obvious way. 
For $Q\in\NN$, we write $\Xi(Q)$ for the subset of $\Xi$ corresponding to $\{(\mu_p)\in\oplus_p\Xi(p)\;|\;\mu_p=0\tx{ for }p\nmid Q\}$. 
Next, we use these results to define the canonical map $i_1:J_{\kap,S}\to M^S_{k+1/2}(\vDe_S)$. 

\begin{lemma}\label{lem:53}
Given $\phi\in J_{\kap,S}$, we put $i_1(\phi)=\phi_0|\til{W}(4b_S)\wp_k$. 
Then $i_1(\phi)$ is an element of $M^S_{k+1/2}(\vDe_S)$. 
\end{lemma}

\begin{proof}
We first note that $\phi_0\in M_{k+1/2}(D_S,\psi_{4b_S})$ by virtue of (\ref{tag:53}) and Proposition \ref{prop:11} (1). 
Let $Q$ be a positive divisor of $D_S$ such that $Q$ and $D_S/Q$ are coprime, and let $(\det S)_Q$ be a positive divisor of $\det S$, the set of prime divisors of which coincides with that of $(\det S, Q)$ and such that $(\det S)_Q$ and $(\det S)/(\det S)_Q$ are coprime. 
From (\ref{tag:53}) we have  
\[\phi_0\|_{k+1/2}\gam_Q\backsim (\det S)_Q^{-1/2}\sum_{\mu\in\Xi(Q)}\phi_\mu, \]
where we write $a\backsim b$ if there is a constant $t$ of absolute value one such that $a=tb$.  
Consequently, $i_1(\phi)$ equals 
\[(4b_S)^{-k/2-1/4}\sum_{\mu\in\Xi(4b_S)}\phi_\mu|\til\del_{4b_S}\wp_k(\tau)=\sum_{\begin{smallmatrix}a\in\NN,\; \mu\in\Xi(4b_S)\\ D_{a,\mu}/\Lam_S\in\frkD_k\end{smallmatrix}}c_\phi(a,\mu)q^{D_{a,\mu}/\Lam_S} \]
up to a constant. 
By Lemma \ref{lem:52}, we can see that $i_1(\phi)$ satisfies (A). 

If $2\notin\frkS_1$, then 
\begin{multline*}
\phi_0|\til{W}(4b_S)\til{Y}(4)
\backsim (\det S)_2^{-1/2}\sum_{\mu\in\Xi(4)}\sum_{j=0}^3\phi_\mu|\til{W}(b_S)\mtrx{1}{j}{0}{1}^*\til{W}(4)\\
\backsim 4(\det S)_2^{-1/2}\phi_0|\til{W}(b_S)\til{W}(4)\backsim 4(\det S)_2^{-1/2}\phi_0|\til{W}(4b_S).  
\end{multline*} 
Note that $(\det S)_2=2$ or $8$ according as $2\in\frkS_0$ or $2\in\frkS_2$. 
Proposition \ref{prop:21} shows that $\phi_0|\til{W}(4b_S)\in M^+_{k+1/2}(D_S)$ if $2\in\frkS_0$, and hence $i_1(\phi)\in M^S_{k+1/2}(\vDe_S)$.  
Lemma \ref{lem:42} (3) shows that $\Tr^S_2(i_1(\phi))=0$ if $2\in\frkS_2$. 

If $p\in\frkS_2$ is odd, then  
\begin{align*}
&\sum_{\mu\in\Xi(4b_S)}\phi_\mu|\til\del_{4b_S}\wp_k\til{W}(p)\til{Y}(p)\\
=&p^{-k/2+3/4}\sum_{\mu\in\Xi(4b_S)}\phi_\mu|\til{W}(p)U(p)\til\del_{4b_S}\wp_k\til{W}(p)\\
=&\sum_{\mu\in\Xi(4b_S)}\sum_{j=0}^{p-1}\phi_\mu|\gam_p^*\mtrx{1}{j}{0}{1}^*\til\del_{4b_S}\wp_k\til{W}(p)\\
=&\gam_{S,p}(1)^{-1}p^{-1}\sum_{\mu\in\Xi(4b_Sp)}\sum_{j=0}^{p-1}\phi_\mu|\mtrx{1}{j}{0}{1}^*\til\del_{4b_S}\wp_k\til{W}(p)\\
=&-\sum_{\mu\in\Xi(4b_S)}\phi_\mu|\til{\del}_{4b_S}\wp_k\til{W}(p). 
\end{align*} 
Here, we use $(\det S)_p=p^2$ and $\gam_{S,p}(1)=-1$. 
Lemma \ref{lem:41} (3) shows that $i_1(\phi)$ satisfies (B). 
Now our proof is complete. 
\end{proof}

To prove the surjectivity of $i_1$, we will construct a natural splitting $j_1$
\[\begin{array}{ccc}
 J^M_{\kap,S}  & \hookrightarrow & J_{\kap,S} \\
               & \nwarrow        & \downarrow i_1 \\
               & j_1        & M^S_{k+1/2}(\vDe_S) 
\end{array}\]
in the following proposition. 

\begin{proposition}\label{prop:51}
For $g\in M^S_{k+1/2}(\vDe_S)$, we let $\phi_0=g|\wp_k^{-1}\til{W}(4b_S)$ and put  
\[j_1(g)(\tau,w)=\sum_{\mu\in\Xi}\phi_\mu(\tau)\vth^S_\mu(\tau,w), \]
where $\phi_\mu$ is defined by (\ref{tag:14}), (\ref{tag:15}). 
Then $j_1(g)\in J^M_{\kap,S}$ and $i_1(j_1(g))$ equals $g$ up to a constant. 
\end{proposition}

\begin{proof}
We follow the same line of computation as in \cite[\S 6]{Kr3}. 
Define $h$ by (\ref{tag:14}). 
We write $a\approx b$ if there is a non-zero constant $t$ such that $a=tb$. 
Note that $h\approx g|\wp_k^{-1}\til{W}(\Lam_S)$ and $h\in M_{k+1/2}(D_S,\psi_{4\Lam_S})$. 

For an integer $r$, let $D_r$ be a positive divisor of $D_S$, the set of prime divisors of which coincides with that of $(r,D_S)$ and such that $D_r$ and $D_S/D_r$ are coprime. 
Since $h|\til{Y}(D_r)\approx h$ by Lemma \ref{lem:41} and \ref{lem:42}, 
\begin{align*}
h\|\mtrx{r}{-1}{D_S}{0}
&\approx h|\til{Y}(D_r)\|\mtrx{r}{-1}{D_S}{0}\\
&\approx\sum_{j=0}^{D_r-1}h\|\mtrx{1}{j}{0}{D_r}\gam_{D_r}\mtrx{D_r}{}{}{1}\mtrx{r}{-1}{D_S}{0}\\
&\approx\sum_{j=0}^{D_r-1}h\|\mtrx{B_{rj}}{-(a+jc)}{B_r}{-cD_r}. 
\end{align*}
Here, we put $B_{rj}=(a+jc)r+(b+jd)D_r^{-1}D_S$ and $B_r=crD_r+dD_S$. 
Since $B_{rj}$ and $B_r$ are coprime, we can choose $\alp\in\Gam_0(D_S)$ and $t_{rj}\in\ZZ$ such that
\[\mtrx{B_{rj}}{-(a+jc)}{B_r}{-cD_r}=\alp\mtrx{1}{t_{rj}}{0}{D_S}\]
We thus have 
\begin{align*}
h\|\mtrx{r}{-1}{D_S}{0}&=\sum_{j=0}^{D_r-1}A_{rjk}h\|\mtrx{1}{t_{rj}}{0}{D_S} \\
&=D_S^{-k/2-1/4}\sum_{j=0}^{D_r-1}\sum_{\nu\in\Xi}\bfe(-t_{rj}S[\nu]/2)A\Pr_{rjk}\phi_\nu\\
&=D_S^{-k/2-1/4}\sum_{\nu\in\Xi} A^{\prime\prime}_{r\nu k}\phi_\nu  
\end{align*}
for suitable constants $A_{rjk}$, $A_{rjk}\Pr$ and $A^{\prime\prime}_{r\nu k}$. 
Therefore we have
\begin{align*}
\phi_\mu\|J&=D_S^{k/2-3/4}\bfa_S(D_{0,\mu})^{-1}\sum_{r=0}^{D_S-1}\bfe(rS[\mu]/2)h\|\mtrx{r}{-1}{D_S}{0}\\
&=D_S^{-1}\bfa_S(D_{0,\mu})^{-1}\sum_{\nu\in\Xi}\sum_{r=0}^{D_S-1}\bfe(rS[\mu]/2)A^{\prime\prime}_{r\nu k}\phi_\nu. 
\end{align*}
It is important to note that the constants $A^{\prime\prime}_{r\mu k}$ are independent of the choice of $g\in M^S_{k+1/2}(D_S)$.

For $\nu\in\Xi$ we put $\vLm_\nu=\{\lam\in\Xi\;|\;S[\lam]/2-S[\nu]/2\in\ZZ\}$. 
By virtue of Lemma \ref{lem:12} and \ref{lem:13}, to prove $\phi\in J_{\kap,S}$, it is enough to show that 
\begin{multline*}
D_S^{-1}\bfa_S(D_{0,\mu})^{-1}\sum_{\lam\in\vLm_\nu}\sum_{r=0}^{D_S-1}\bfe(rS[\mu]/2)A^{\prime\prime}_{r\lam k}\\
=(\det S)^{-1/2}\bfe(n/8)\sum_{\lam\in\vLm_\nu}\bfe(-S(\mu,\lam))  
\end{multline*}
for all $\mu\in\Xi$. 
As one can see easily, $A^{\prime\prime}_{r\mu k\Pr}$ does not depend on $k\Pr$, as long as $k\Pr\equiv k\pmod 8$, so we may suppose that $k$ is sufficiently large.  
The Jacobi Eisenstein series $E_{\kap,S}$ is an element of $J^M_{\kap,S}$ and each theta component is non-zero by Corollary \ref{cor:81}. 
We can thus observe that this equality holds, replacing $g$ by $i_1(E_{\kap,S})$. 
Once we know that $j_1(g)\in J_{\kap,S}$, the remaining parts follow immediately. 
\end{proof}

For convenience, we list the Fourier-coefficient formulas of $\vPh$ in a more concrete form. 

\begin{lemma}\label{lem:54}
Let $(a,\alp)\in\calt^+$, $\frkf$ a non-negative integer and $\frkd$ a positive integer such that $(-1)^k\frkd$ is a fundamental discriminant. 
Notation and assumption being as in Theorem \ref{thm:31}, we have the following formulas.  
\begin{enumerate}
\renewcommand\labelenumi{(\theenumi)}
\item If $p\in\frkS_0$, then 
\[c_g(\frkd)p^{(k-1/2)\frkf}l_{p,S,\frkd p^{2\frkf}}(\alp_p)=c_g(\frkd p^{2\frkf}). \]
\item If $p$ is a prime divisor of $b^{-1}b_S$, then 
\[c_g(\frkd)p^{(k-1/2)\frkf}l_{p,S,\frkd p^{2\frkf}}(\alp_p)=c_g(\frkd p^{2\frkf})+\eta_p(S)p^kc_g(\frkd p^{2\frkf-2}). \]
\item If $p$ is a prime divisor of $d^{-1}d_S$, then 
\begin{multline*}
c_g(\frkd)p^{(k-1/2)\frkf}l_{p,S,\frkd p^{2\frkf}}(\alp_p)=\\
c_g(\frkd p^{2\frkf})-\biggl(\frac{(-1)^k\frkd p^{2\frkf-2}}{p}\biggl)p^kc_g(\frkd p^{2\frkf-2})-p^{2k}c_g(\frkd p^{2\frkf-4}).  
\end{multline*}
\item If $p$ is a prime divisor of $b$, then 
\begin{multline*}
l_{p,S,\Del_{a,\alp}}(-\eta_p(S)p^{-1/2})=\\
\bfa_{S,p}(D_{a,\alp})^{-1}(-\eta_p(S)p^{-1/2})^{\frkf_p(\Del_{a,\alp})}
(1+\eta_p(S)\Psi_p((-1)^k\Del_{a,\alp})). 
\end{multline*}
\item If $p$ is a prime divisor of $d$ and $\eps\in\{\pm 1\}$, then
\[l_{p,S,\Del_{a,\alp}}(\eps p^{-1/2})=\frac{(p+1)}{\bfa_{S,p}(D_{a,\alp})}(\eps p^{-1/2})^{\frkf_p(\Del_{a,\alp})}(1-\eps\Psi_p((-1)^k\Del_{a,\alp})). \]
\end{enumerate}
\end{lemma}

\begin{proof}
We can prove Lemma \ref{lem:54} by a case by case calculation. 
\end{proof}

\begin{lemma}\label{lem:55}
Under the notation as in Theorem \ref{thm:31}, the function $\vPh$ coincides with $j_1(g^*)$ up to a constant. 
\end{lemma}

\begin{proof}
Put $h=g^*|\wp_k^{-1}\til{W}(\Lam_S)$. 
We can see that 
\[j_1(g^*)(\tau,w)\approx\sum_{(a,\alp)\in\calt^+}\bfa_S(D_{a,\alp})^{-1}c_h(D_{a,\alp})q^a\bfe(S(\alp,w)) \]
from the definition of $j_1$. 
Let $p$ be a odd prime divisor of $d^{-1}d_S$. 
If $p$ is odd, then (\ref{tag:41}) shows that $g|Q(p)\til{W}(p)$ equals 
\[\sum_{m\in\NN}\bfa_{S,p}(m)\biggl(c_g(pm)-\biggl(\frac{(-1)^kp^{-1}m}{p}\biggl)p^kc_g(p^{-1}m)-p^{2k}c_g(p^{-3}m)\biggl)q^m \]
up to scalar multiple, and if $p=2$, then 
\begin{multline*}
g|Q(2)\wp_k^{-1}=3g|U(4)-2g|\til{T}(p^2)\\
=\sum_{m\in\NN}\bfa_{S,2}(m)\biggl(c_g(4m)-\biggl(\frac{(-1)^km}{2}\biggl)2^kc_g(m)-2^{2k}c_g(4^{-1}m)\biggl)q^m. 
\end{multline*}

Let $q$ be a prime divisor of $d$. 
It follows from Lemma \ref{lem:21} that if $q$ is odd, then $g|\til{W}(q)\approx g|U(q)$, and if $q=2$, then 
\[g|\wp_k^{-1}\approx \sum_{m\in\NN}c_g(4m)q^m. \]
The information obtained from these calculations, together with Remark \ref{rem:42} and Lemma \ref{lem:54}, proves the relation we want. 
\end{proof}

Finally, we complete the proof of Theorem \ref{thm:31}. 
As the above Lemma shows, $\vPh$ is an element of $J^{\cusp,M}_{\kap,S}$. 
Taking Corollary \ref{cor:81} into account, we can prove that $\vPh$ is a Hecke eigenform by the same type of arguments as in the proof of \cite[Theorem 3.3]{Ik2}. 
Proposition \ref{prop:22}, \ref{prop:41} and \ref{prop:51} show that the correspondence $f\mapsto\vPh$ is bijective. 
The assertion concerning the $L$-function directly follows from Proposition \ref{prop:12}. 

%%%%%%%%%%%%%%%%%%%%%%%%%%%%%%%%%%%%%%%%%%%%%%%%%%%%%%%%%%%%%%%%%%%%%%%%%

\section{\bf The space $M^S_k(D_S, \chi_S)$}\label{sec:6} 

We now discuss the case when $n$ is even. 
We first summarize the details of the analogous work in \S \ref{sec:4}. 

For a positive divisor $Q$ of $D_S$ such that $Q$ and $D_S/Q$ are coprime, we define $W(Q)$ and $Y(Q)$ by replacing $k+1/2$ by $k$ and $\|$ by $|$ in the definitions of $\til{W}(Q)$ and $\til{Y}(Q)$ in \S \ref{sec:2}. 
Put $Y\Pr(p)=Y(p^{\ord_pD_S})$. Set
\[\Tr^S_p(f)=(p+1)^{-1}\sum_{\gam\in\Gam_0(D_S)\bsl\Gam_0(p^{-1}D_S)}\chi_S(\gam)f|_k\gam\]
for $f\in M_k(\Gam_0(D_S),\chi_S)$ and each prime divisor $p$ of $d_S$. 

For each rational prime $p$, the $p$-primary component $\chi_{S,p}$ of $\chi_S$ is defined by 
\[\chi_{S,p}(m)
=\begin{cases}\chi_S(m\Pr) &\tx{if $p\nmid m$, }\\
0 &\tx{if $p|m$, }
\end{cases}\] 
where $m\Pr$ is an integer such that 
\[m\Pr\equiv \begin{cases}m &\pmod {p^{\ord_p\frkd_S}}. \\
1 &\pmod {p^{-\ord_p\frkd_S}\frkd_S}. \end{cases}\] 
One should not confuse $\chi_{S,p}$ with $\Chi_{S,p}$. 

\begin{lemma}\label{lem:61}
For a given $f\in M_k(\Gam_0(D_S), \chi_S)$ and for $p\in\frkS_1$ and $q\in\frkS_2$, we consider the following conditions: 
\begin{enumerate}
\renewcommand\labelenumi{(\theenumi)}
\item[(\roman{one})] $c_f(m)=0$ if $\chi_{S,p}((-1)^km)=-\eta_p(S/2)$;  
\item[(\roman{two})] $\Tr^S_q(f)=0$;
\item[(a)] $\vep_p(\chi_S)^{-1}p^{-(\ord_pD_S)/2}f|Y\Pr(p)=\eta_p(S/2)f$;  
\item[(b)] $f|W(q)Y\Pr(q)=-f|W(q)$. 
\end{enumerate} 
Here we put 
\[\vep_p(\chi_S)=\begin{cases}
1 &\tx{if $\chi_{S,p}(-1)=1$. }\\
\iu &\tx{if $\chi_{S,p}(-1)=-1$. }
\end{cases}\]
Then the following assertions holds. 
\begin{enumerate}
\renewcommand\labelenumi{(\theenumi)}
\item $\vep_p(\chi_S)p^{-(\ord_pD_S)/2}Y\Pr(p)$ is an involution on $M_k(\Gam_0(D_S),\chi_S)$. 
\item The conditions (\roman{one}) and (a) are equivalent. 
\item The conditions (\roman{two}) and (b) are equivalent. 
\end{enumerate}
\end{lemma}

\begin{proof}
Concerning the second assertion, Proposition 5 of \cite{Kr2} handled the case in which $S/2$ is the norm form of an imaginary quadratic field. 
The proof in \cite{Kr2} goes over with only minor changes, and the first assertion easily follows.   
The proof of our assertion (3) follows that of Lemma \ref{lem:41} (3) closely. 
\end{proof}

\begin{definition}\label{def:61}
The space $M^S_k(D_S,\chi_S)$ consists of all functions $f\in M_k(\Gam_0(D_S),\chi_S)$ with the following properties:  
\begin{enumerate}
\renewcommand\labelenumi{(\theenumi)}
\item[(A)] $f$ satisfies (\roman{one}) of Lemma \ref{lem:61} for every prime $p\in\frkS_1$;  
\item[(B)] $f$ satisfies (\roman{two}) of Lemma \ref{lem:61} for every prime $q\in\frkS_2$. 
\end{enumerate}
We put $S^S_k(D_S,\chi_S)=S_k(\Gam_0(D_S),\chi_S)\cap M^S_k(D_S,\chi_S)$. 
\end{definition}

Fix a positive divisor $d$ of $d_S$. 
Fix $f\in\Prm_k(\frkd_Sd,\chi_S)$. 
For a subset $P$ of $\frkS_1\cup\frkS_2$, we put
\begin{gather*}
\chi_{S,P}=\prod_{p\in P}\chi_{S,p}, \quad\quad 
\chi\Pr_{S,P}=\prod_{p\notin P}\chi_{S,p}, \quad\quad
\eta_P(S/2)=\prod_{p\in P}\eta_p(S/2), \\
\frkd_{S,P}=\prod_{p\in P}p^{\ord_p\frkd_S}, \quad\quad
D_{S,P}=\prod_{p\in P}p^{\ord_pD_S}.   
\end{gather*}
Recall that there exists  
\[f_P(\tau)=\sum_mb(m)q^m\in \Prm_k(\frkd_Sd,\chi_S)\]
such that
\[b(p)=\begin{cases}
\chi_{S,P}(p)c_f(p) &\tx{if $p\notin P$. }\\
\chi\Pr_{S,P}(p)\overline{c_f(p)} &\tx{if $p\in P$. }
\end{cases}\]
Clearly, $f_P=f$ if $P\subset\frkS_2$. 
Note that $f_P$ is equal to $f|W(D_{S,P})$ up to scalar multiple by \cite[Theorem 4.6.16]{Mi}. 
Following \cite{Ik3}, we put
\[f^\sim=\sum_{P\subset\frkS_1}\eta_P(S/2)\chi_{S,P}((-1)^kd^{-1}d_S)f_P. \]

We write $S^\new_k(\frkd_Sd,\chi_S)$ for the space of newforms for $S_k(\Gam_0(\frkd_Sd),\chi_S)$, and define  
\begin{multline*}
\frkS^S_k(\frkd_Sd,\chi_S)=\{h\in S^\new_k(\frkd_Sd,\chi_S)\; |\; c_h(m)=0\\
\tx{ if there is } p\in\frkS_1\tx{ such that }\chi_{S,p}((-1)^kd^{-1}d_Sm)=-\eta_p(S/2)\}. 
\end{multline*} 
The proof of the following three assertions are the same as those of Lemma 15.4, Corollary 15.5 and Proposition 15.17 in \cite{Ik3}, and hence are omitted. 

\begin{lemma}\label{lem:62}
\begin{enumerate}
\renewcommand\labelenumi{(\theenumi)}
\item The $m$-th Fourier coefficient of $f_P$ is equal to
\[c_f(m\Pr_Pm\Pr)\oln{c_f(m_P)}\prod_{p\in P}\Chi_{S,p}(m), \]
where we put 
\begin{align*}
m_p&=p^{\ord_pm}, & 
m_P&=\prod_{p\in P}m_p, &
m\Pr_P&=\prod_{p\in\frkS_1-P}m_p, & 
m\Pr&=\prod_{p\notin\frkS_1}m_p. 
\end{align*}  
\item The $m$-th Fourier coefficient of $f^\sim$ is equal to
\[\prod_{p\in\frkS_1}(c_f(m_p)+\eta_p(S/2)\Chi_{S,p}((-1)^kd^{-1}d_Sm)\oln{c_f(m_p)})c_f(m\Pr). \] 
In particular, we have $f^\sim\in\frkS^S_k(\frkd_Sd,\chi_S)$. 
\item 
The space $\frkS^S_k(\frkd_Sd,\chi_S)$ is spanned by the $\CC$-linear combinations of $\{h^\sim\;|\; h\in \Prm_k(\frkd_Sd,\chi_S)\}$. 
\end{enumerate}
\end{lemma}

The operator $Q(p): M_k(\Gam_0(\frkd_Sd),\chi_S)\to M_k(\Gam_0(\frkd_Spd),\chi_S)$ is defined by 
\[h|Q(p)=h|T(p)-\xi_p(S)(p+1)p^{k/2-1}h|_k\mtrx{p}{}{}{1} \]
for each prime divisor $p$ of $d^{-1}d_S$. 
We put $f^*=f^\sim\prod_{p|d^{-1}d_S}Q(p)$. 

\begin{proposition}\label{prop:61}
Notation and assumption being as above, 
\[f^*\in S^S_k(D_S,\chi_S). \]
Moreover, 
\[S^S_k(D_S,\chi_S)=\oplus_{d\geq 1,\; d|d_S}\La h^*\;|\; h\in \Prm_k(\frkd_Sd,\chi_S)\Ra. \]
\end{proposition}

\begin{proof}
For each prime divisor $p$ of $d^{-1}d_S$, we have  
\begin{align*} 
\Tr^S_p(f|Q(p))
&=f|T(p)-\xi_p(S)p^{k/2-1}(p+1)\Tr^S_p\Big(f|_k\mtrx{p}{}{}{1}\Big)\\
&=f|T(p)-f|T(p)=0. 
\end{align*}

Since $f$ is a newform, $\Tr^S_p(f^*)=0$ for each prime divisor $p$ of $d$. 
It is easily to see that 
\[f|Q(p)(\tau)=\sum_m(c_f(pm)-\xi_p(S)p^kc_f(p^{-1}m))q^m, \]
which combined with Lemma \ref{lem:62} (2) shows that $f^*$ satisfies the condition (A), and hence $f^*\in S^S_k(D_S,\chi_S)$. 

The same process as in the proof of Proposition \ref{prop:41} concludes that 
\begin{multline*}
\{h\in S_k(\Gam_0(D_S),\chi_S)\;|\; \Tr^S_p(h)=0 \tx{ for every }p\in\frkS_2\}\\
=\oplus_{d\geq1,\;d|d_S}S^\new_k(\frkd_Sd,\chi_S)|\textstyle\prod_{p|d^{-1}d_S}Q(p).  
\end{multline*} 
Lemma \ref{lem:62} (3) now completes the proof of Proposition \ref{prop:61}. 
\end{proof}

Notice that $\chi_{S,P}$ is the primitive Dirichlet character corresponding to the quadratic field
\[K_P=\QQ((\chi_{S,P}(-1)\frkd_{S,P})^{1/2}), \]
provided that $P$ is a non-empty subset of $\frkS_1$. 
\begin{lemma}\label{lem:63}
Let $f$ be as before. 
The following conditions are equivalent:
\begin{enumerate}
\renewcommand\labelenumi{(\theenumi)} 
\item[(\roman{one})] $f^*=0$; 
\item[(\roman{two})] $f$ comes from a Hecke character of $K_P$ for some $\emptyset\neq P\subset\frkS_1$ such that $\chi_{S,P}((-1)^kd^{-1}d_S)=-\eta_P(S/2)$. 
\end{enumerate}
\end{lemma}

\begin{proof}
The proof proceeds like that of \cite[Corollary 15.6]{Ik2}. 
It is immediate that $f^*=0$ if and only if $f^\sim=0$. 
If $f^\sim=0$, then we must have $f=f_P$ and $\chi_{S,P}((-1)^kd^{-1}d_S)=-\eta_P(S/2)$ for some non-empty subset $P$ of $\frkS_1$ since distinct primitive forms are linearly independent. 
Conversely, if $f=f_P$ and $\chi_{S,P}((-1)^kd^{-1}d_S)=-\eta_P(S/2)$, then 
\[f^\sim=(f_P)^\sim=\eta_P(S/2)\chi_{S,P}((-1)^kd^{-1}d_S)f^\sim, \]
and hence $f^\sim=0$. 
In order that $f=f_P$ for $\emptyset\neq P\subset \frkS_1$, it is necessary and sufficient that $f$ comes from a Hecke character of $K_P$ (see \cite[Appendix C]{GL}). 
\end{proof}

\begin{definition}\label{def:62}
The subset $V_k(\frkd_Sd,\chi_S)$ of $\Prm_k(\frkd_Sd,\chi_S)$ consists of all primitive forms which satisfy the condition (\roman{two}) of Lemma \ref{lem:63}. 
\end{definition}

\begin{remark}\label{rem:61}
Let us note that local components of automorphic representations coming from Hecke characters of quadratic fields are principal series or supercuspidal at nonarchimedean places.  
In particular, $V_k(\frkd_Sd,\chi_S)$ is empty unless $d=1$. 
\end{remark}

%%%%%%%%%%%%%%%%%%%%%%%%%%%%%%%%%%%%%%%%%%%%%%%%%%%%%%%%%%%%%%%%%%%%%%%%%

\section{\bf Proof of Theorem \ref{thm:32}}\label{sec:7}

We now focus on the proof of Theorem \ref{thm:32}. 
The methods of this section are substantially those of \S \ref{sec:5}, so that we will sometimes omit details. 
For each rational prime $p$ and $\ell\in\ZZ_p$, we put
\[\bfa_{S,p}(\ell)=\begin{cases}
1 &\tx{if $p\in\frkS_0$. }\\
1+\eta_p(S/2)\chi_{S,p}((-1)^kd_S\ell) &\tx{if $p\in\frkS_1$. }\\
p\del(p\nmid \ell)+1 &\tx{if $p\in\frkS_2$. }
\end{cases}\] 
Recall that $k$ is a positive integer such that $\kap=k+\tfrac{n}{2}$ is even. 

\begin{lemma}\label{lem:71}
Let $\ell\in\NN$. The following conditions are equivalent: 
\begin{enumerate}
\renewcommand\labelenumi{(\theenumi)}
\item[(\roman{one})] all primes $p$ satisfy $\bfa_{S,p}(\ell)\neq 0$;  
\item[(\roman{two})] there exists $(a,\alp)\in\calt^+$ such that $\ell=D_{a,\alp}$. 
\end{enumerate}
Moreover, $\bfa_S(\ell)=\prod_p\bfa_{S,p}(\ell)$, where $\bfa_S(\ell)$ is defined in Lemma \ref{lem:14}.  
\end{lemma}

\begin{proof}
We define $Q$, $U$, $Z_p$, $M_p$ and $M^*_p$ as in the proof of Lemma \ref{lem:52}. 
By the same principle, it is enough to prove 
\[\bfa_{S,p}(\ell)=\sharp\{\alp\in M_p^*/M_p\;|\;Q[\alp]\equiv \ell\pmod {D_S}\}\]
for every prime $p$ and $\ell\in\ZZ_p$. 
Put $\calk_p=\QQ_p(((-1)^k\det S)^{1/2})$. 
It is significant that the isometry class of $Q$ over $\QQ_p$ is determined by $n$, $\calk_p$ and $\eta_p(S/2)$. 
We thus proceed according to the nature of $\calk_p$ and the sign $\eta_p(S/2)$. 

\begin{enumerate}
\renewcommand\labelenumi{(\theenumi)}
\item[(1)] If $\calk_p=\QQ_p$ and $\eta_p(S/2)=1$, then $Z_p=0$ and $p\in\frkS_0$. 
\item[(2)] Suppose that $\calk_p=\QQ_p$ and $\eta_p(S/2)=-1$. 
Then $Z_p$ is isometric to $(H,-D_S\cdot\nu)$ and 
\begin{align*}
p&\in\frkS_2, & M_p&=\calo, & M_p^*&=\frkP^{-1}. 
\end{align*}
Here, $\nu$, $\calo$ and $\frkP$ are as in the proof of Lemma \ref{lem:52}. 
\item[(3)] Next suppose that $\calk_p/\QQ_p$ is a quadratic extension and $\eta_p(S/2)=1$. 
We denote by $\Nr$ the field theoretic norm of $\calk_p/\QQ_p$. 
Since $(-1)^{k+1}\frkd_S\in\Nr(\calk_p^\times)$, 
\[Z_p\sim(\calk_p, -D_S\cdot\Nr)\sim(\calk_p, (-1)^kd_S\cdot\Nr), \]
where $\sim$ denotes equivalence up to isometry. 
We thus see that   
\begin{align*}
M_p&=\frkd, & M_p^*&=\frkr, 
\end{align*}
letting $\frkr$ be the integer ring of $\calk_p$ and $\frkd$ the different of $\calk_p$ over $\QQ_p$. 
Note that $s_p(S)\leq 1$. 
\item[(4)] It remains to check the case in which $\calk_p/\QQ_p$ is a quadratic extension and $\eta_p(S/2)=-1$. 
Fix some element $\alp\in Q^\times_p$ such that 
\begin{align*}
\alp&\notin\Nr(\calk_p^\times), & 
\alp&\in\begin{cases}
p\ZZ_p^\times &\tx{if $\calk_p/\QQ_p$ is unramified. }\\
\ZZ_p^\times &\tx{if $\calk_p/\QQ_p$ is ramified. }
\end{cases}
\end{align*}  
Then we have 
\begin{align*}
Z_p&\sim(\calk_p, (-1)^kd_S\alp\cdot\Nr), & 
s_p(S)&=1+\ord_p\alp, \\ 
M_p&=\begin{cases}
\frkd &\tx{if $p\in\frkS_1$, }\\
\frkr &\tx{if $p\in\frkS_2$, } 
\end{cases} &
M_p^*&=\begin{cases}
\frkr &\tx{if $p\in\frkS_1$. }\\
p^{-1}\frkr &\tx{if $p\in\frkS_2$. }
\end{cases}
\end{align*}
\end{enumerate}

Now Lemma \ref{lem:71} easily follows from the description above. 
\end{proof}

We define the symplectic vector space $\WW$, the metaplectic cover $\wtl{\Sp(\WW)}$, $\Xi(Q)$ and $(\det S)_Q$ as in \S \ref{sec:5}. 
Since we are assuming that $n=\dim X$ is even, there is a splitting $\SL_2(\AA)\to\wtl{\Sp(\WW)}$. 
When the local Weil representation $\ome_v$ of $\wtl{\Sp(\WW)(\QQ_v)}$, associated to $\bfe_v$, is realized on $\cals(X_v)$ in the usual Schr\"{o}dinger model, composing $\ome_v$ with this splitting, we obtain a representation of $\SL_2(\QQ_v)$ on $\cals(X_v)$. 
This representation is given by the following formulas: 
\begin{gather*}
\ome_v\Big(\mtrx{a}{b}{0}{a^{-1}}\Big)l(\xi)=\Chi_{S,v}(a)|a|_v^{n/2}\bfe_v(abS[\xi]/2)l(\xi a), \\
\ome_v(J)(\xi)=\gam_{S,v}(1)^{-1}\int_{X_v}l(\eta)\bfe_v(-S(\eta, \xi))d\eta
\end{gather*}
We define the Weil representation $\ome_\bff$ of $\SL_2(\AAf)$ on $\cals(X_\bff)$ by the restricted tensor product of the local Weil representations. 

\begin{lemma}\label{lem:72}
Given $\phi\in J_{\kap,S}$, we put $i_1(\phi)=\phi_0|W(\frkd_S)$. 
Then $i_1(\phi)$ is an element of $M^S_k(D_S,\chi_S)$. 
\end{lemma}

\begin{proof}
For $\gam\in\SL_2(\ZZ)$, we define $u(\gam)_{\mu\nu}\in\GL_d(\CC)$ $(d=\sharp\Xi)$ by  
\[\ome_\bff(\gam^{-1})l_\nu=\sum_{\mu\in\Xi}u(\gam)_{\mu\nu}l_\mu. \]
Following \S \ref{sec:5}, we have 
\[\phi_\nu|_k\gam=\sum_{\nu\in\Xi}\oln{u(\gam)_{\mu\nu}}\phi_\mu. \]
Since $u$ is a unitary matrix, we have $\phi_0\in M_k(\Gam_0(D_S),\chi_S)$ by Proposition \ref{prop:11} (2). 
Let $Q$ be a positive divisor of $D_S$ such that $Q$ and $D_S/Q$ are coprime. 
Substituting $\gam_Q$ for $\gam$, we have  
\[\phi_0|_k\gam_Q=\gam_{S,Q}(1)^{-1}\cdot(\det S)_Q^{-1}\sum_{\mu\in\Xi(Q)}\phi_\mu, \]
where $\gam_{S,Q}(1)=\prod_{p|Q}\gam_{S,p}(1)$. 
This implies that $i_1(\phi)$ is equal to 
\[\frkd_S^{-k/2}\sum_{\mu\in\Xi(\frkd_S)}\phi_\mu|_k\mtrx{\frkd_S}{}{}{1}(\tau)
=\sum_{a\in\NN,\;\mu\in\Xi(\frkd_S)}c_\phi(a,\mu)q^{D_{a,\mu}/d_S} \]
up to a constant. 
Now the proof follows that of Lemma \ref{lem:53}, with the obvious modifications.  
\end{proof}

\begin{proposition}\label{prop:71}
For $f\in M^S_k(D_S,\chi_S)$, we put $\phi_0=f|W(\frkd_S)$ and define $j_1(f)$ by 
\[j_1(f)(\tau,w)=\sum_{\mu\in\Xi}\phi_\mu(\tau)\vth^S_\mu(\tau,w), \]
where $\phi_\mu$ is defined by (\ref{tag:14}), (\ref{tag:15}). 
Then $j_1(f)\in J^M_{\kap,S}$ and $i_1(j_1(f))$ equals $f$ up to a constant. 
\end{proposition}

\begin{proof}
The whole picture in the proof of Proposition \ref{prop:51} works out perfectly well when $n$ is even, so we omit details.  
\end{proof}

\begin{lemma}\label{lem:73}
Let $(a,\alp)\in\calt^+$ and $m$ a non-negative integer. 
Under the notation as in Theorem \ref{thm:32}, we have the following formulas.  
\begin{enumerate}
\renewcommand\labelenumi{(\theenumi)}
\item If $p\in\frkS_0$, then 
\[p^{m(k-1)/2}l_{p,S,p^m}(\alp_p)=c_f(p^m). \]
\item If $p\in\frkS_1$, then 
\begin{multline*}
p^{(k-1)(\ord_pD_{a,\alp})/2}l_{p,S,D_{a,\alp}}(\alp_p)=\\
\bfa_{S,p}(D_{a,\alp})^{-1}(c_f(p^{\ord_pD_{a,\alp}})+\eta_p(S/2)\Chi_{S,p}((-1)^kd_SD_{a,\alp})\oln{c_f(p^{\ord_pD_{a,\alp}})}).  
\end{multline*} 
\item If $p$ is a prime divisor of $d^{-1}d_S$, then 
\[p^{m(k-1)/2}l_{p,S,p^m}(\alp_p)=c_f(p^m)-\xi_p(S)p^kc_f(p^{m-2}). \]
\item If $p$ is a prime divisor of $d$, then 
\[l_{p,S,p^m}(\alp_p)=(p+1)\bfa_{S,p}(p^m)^{-1}(\xi_p(S)\alp_p)^m. \]
\end{enumerate}
\end{lemma}

\begin{remark}\label{rem:71}
There is a simple relation  
\[\eta_p(S_{a,\alp})=\eta_p(S/2)\Chi_{S,p}((-1)^kd_SD_{a,\alp}). \]
\end{remark}

\begin{proof}
Theorem 4.6.17 of \cite{Mi} tells us that if $p\in\frkS_1$, then $|\alp_p|=1$, and if $d$ is divisible by $p$, then $\alp_p\in\{\pm(\xi_p(S))^{1/2}p^{-1/2}\}$. 
These formulas can be verified by an immediate computation. 
\end{proof}

\begin{lemma}\label{lem:74}
With the notation of Theorem \ref{thm:32}, the function $\vPh$ coincides with $j_1(f^*)$ up to a constant. 
\end{lemma}

\begin{proof}
Put $h=f^*|W(d_S)$. 
We can see that
\[j_1(f^*)(\tau,w)\approx\sum_{(a,\alp)\in\calt^+}\bfa_S(D_{a,\alp})^{-1}c_h(D_{a,\alp})q^a\bfe(S(\alp,w)) \]
from the definition of $j_1$. 
Recall that $\approx$ denotes equivalence up to scalar multiple. 

If $p$ is a prime divisor of $d^{-1}d_S$, then 
\begin{align*}
f|Q(p)W(p)
&=f|T(p)\mtrx{p}{}{}{1}-(p+1)p^{k/2-1}f\\
&\approx\sum_m\bfa_{S,p}(m)(c_f(m)-\xi_p(S)p^kc_f(p^{-2}m))q^a. 
\end{align*}
Corollary 4.6.18 of \cite{Mi} implies that
\[f_P|W(d)=(-1)^\ell d^{1-k/2}\oln{c_{f_P}(d)}f_P=(-1)^\ell d^{1-k/2}\chi_{S,P}(d)\oln{c_f(d)}f_P, \]
where $\ell$ denotes the number of distinct prime factors of $d$. 
Therefore, putting $\bfa_{S,\frkd_S}(m)=\prod_{p\in\frkS_1}\bfa_{S,p}(m)$, we have 
\begin{align*}
&f^\sim|W(d)(\tau)\approx\sum_{P\subset\frkS_1}\eta_P(2^{-1}S)\chi_{S,P}((-1)^kd_S)f_P(\tau)\\
\approx&\sum_m\prod_{p\in\frkS_1}(c_f(m_p)+\eta_p(S/2)\Chi_{S,p}((-1)^kd_Sm)\oln{c_f(m_p)})c_f(m\Pr)q^m\\
\approx&\sum_m\bfa_{S,\frkd_S}(m)\prod_{\begin{smallmatrix} p\in\frkS_1 \\ p|m\end{smallmatrix}}(c_f(m_p)+\eta_p(S/2)\Chi_{S,p}((-1)^kd_Sm)\oln{c_f(m_p)})c_f(m\Pr)q^m 
\end{align*}
(see Lemma \ref{lem:62} (2)). 
Here, if $m$ and $p$ are coprime, then we can show $\chi_{S,p}(m)=\Chi_{S,p}(m)$ as in \cite[Lemma 15.1]{Ik3}. 
Comparing these calculations with Lemma \ref{lem:73}, we can observe that $\vPh$ agrees with $j_1(f^*)$ up to a constant. 
\end{proof}

Next is the proof of Theorem \ref{thm:32}. 
By Lemma \ref{lem:63} and \ref{lem:74}, $\vPh$ is a non-zero element of $J^{\cusp,M}_{\kap,S}$. 
The proof that $\vPh$ is a Hecke eigenform proceeds like that of \cite[Theorem 3.3]{Ik2}. 
Proposition \ref{prop:12} gives rise to an explicit description of its $L$-function. 
As we have observed in Proposition \ref{prop:61} and \ref{prop:71}, the space $J^{\cusp,M}_{\kap,S}$ is spanned by these Jacobi forms. 
Suppose that two primitive forms $f\in\Prm^*_k(\frkd_Sd,\chi_S)$ and $f\Pr\in \Prm^*_k(\frkd_Sd\Pr,\chi_S)$ produce the same Hecke eigenform in $J^{\cusp,M}_{\kap,S}$. 
Proposition \ref{prop:71} shows that $f^*\approx f^{\prime*}$. 
Clearly, $d=d\Pr$ and $f^\sim\approx f^{\prime\sim}$. 
Since distinct primitive forms are linearly independent, there is $P\subset\frkS_1$ such that $f=f\Pr_P$. 
We have thus completed our proof. 

In the reminder of this section, we do not require $n$ to be even. 

The Petersson inner products on $J^\cusp_{\kap,S}$ are defined by
\[\La\phi, \psi\Ra=\int_{\vGm\bsl\cald}
\phi(\tau,w)\oln{\psi(\tau,w)}y^{\kap-2} e^{-2\pi yS[\xi]}dxdyd\xi d\eta\]
for $\phi$, $\psi\in J^\cusp_{\kap,S}$. 
Here, let $\tau=x+\iu y$, $w=\tau\xi+\eta$ and $dx$, $dy$ (resp. $d\xi$, $d\eta$) the Lebesgue measures on $\RR$ (resp. $X_\infty$). 

\begin{lemma}\label{lem:75}
Let $\phi$ and $\psi$ be elements of $J^\cusp_{\kap,S}$. 
We have
\[\La \phi,\psi\Ra=2^{-1-n/2}(\det S)^{-1/2}\sum_{\mu\in\Xi}\La \phi_\mu,\psi_\mu\Ra. \]
\end{lemma}

\begin{proof}
It is well-known that 
\begin{multline*}
\int_{(L+L\tau)\bsl X\otimes_\QQ\CC}\vth^S_\mu(\tau,w)\oln{\vth^S_\nu(\tau,w)}
e^{-2\pi yS[\xi]}d\xi d\eta\\
=\begin{cases}2^{-n/2}(\det S)^{-1/2}y^{n/2} &\tx{if $\mu=\nu$. }\\
0&\tx{if $\mu\neq\nu$. }
\end{cases}
\end{multline*}
We can prove Lemma \ref{lem:75} by the same way as \cite[Proposition 17.1]{Ik3}. 
\end{proof}

\begin{proposition}\label{prop:72}
Let $\kap$ be even. Put 
\begin{align*}
J^0_{\kap,S}&=\{\phi\in J_{\kap,S}\;|\;\phi_0\tx{ identically zero}\}, & 
J^{\cusp,0}_{\kap,S}&=J^0_{\kap,S}\cap J^\cusp_{\kap,S}. 
\end{align*}
Then $J^{\cusp,M}_{\kap,S}$ and $J^{\cusp,0}_{\kap,S}$ are closed under the action of Hecke operators, and there are the following decompositions: 
\begin{align*}
J_{\kap,S}&=J^M_{\kap,S}\oplus J^0_{\kap,S}, & 
J^\cusp_{\kap,S}&=J^{\cusp,M}_{\kap,S}\oplus J^{\cusp,0}_{\kap,S}. 
\end{align*}
\end{proposition}

\begin{proof}
We have already seen the last assertion in Proposition \ref{prop:51} and \ref{prop:71}. 
Our first task is to show that $J^{\cusp,M}_{\kap,S}$ and $J^{\cusp,0}_{\kap,S}$ are orthogonal. 

To see this, let $\phi\in J^{\cusp,M}_{\kap,S}$ and $\psi\in J^{\cusp,0}_{\kap,S}$. 
Applying Lemma \ref{lem:12} to the present setting,  we have 
\[\sum_{\mu\in\Xi}\psi_\mu=(\det S)^{1/2}\bfe(-n/8)\psi_0\|_{\kap-n/2}J=0.\] 
Comparing the Fourier coefficients, we can see $\sum_{\mu\in\vLm_\nu}\psi_\mu\equiv 0$ ($\vLm_\mu$ is defined in the proof of Proposition \ref{prop:51}). 
We now have 
\[\La \phi, \psi\Ra\approx\sum_{\mu\in\Xi}\La\phi_\mu, \psi_\mu\Ra\approx\sum_\nu\sum_{\mu\in\vLm_\nu}\La\phi_\nu, \psi_\mu\Ra=0, \]
using Lemma \ref{lem:13} and \ref{lem:75}. 
Since $J^{\cusp,M}_{\kap,S}$ is Hecke-invariant, so is $J^{\cusp,0}_{\kap,S}$. 
\end{proof}

%%%%%%%%%%%%%%%%%%%%%%%%%%%%%%%%%%%%%%%%%%%%%%%%%%%%%%%%%%%%%%%%%%%%%%%%%

\section{\bf Fourier coefficients of Jacobi Eisenstein series}\label{sec:8}

To complete the picture, we must show that the Jacobi Eisenstein series $E_{\kap,S}$ is an element of $J^M_{\kap,S}$. 
We here obtain an explicit formula for the Fourier-coefficients of $E_{\kap,S}$. 
The computation is based on that of Sugano's paper \cite{Su}. 

The letter $\kap$ stands for an even integer throughout this section. 
Put
\begin{align*}
\vGm_\infty&=\{\gam\in\vGm \;|\; j_\kap(\gam, (\tau, w))\equiv 1\}\\
&=\Big\{\pm\mtrx{1}{\ell}{0}{1}[0, \eta, \zet]\;\Big|\; \ell,\;\zet\in\ZZ,\; \eta\in L\Big\}. 
\end{align*}
The Eisenstein series on the Jacobi group is defined, for $\kap>n+2$, by the series
\[E_{\kap,S}(\tau,w)=\sum_{\gam\in\vGm_\infty\bsl\vGm}j_\kap(\gam, (\tau,w))^{-1}, \] 
and by analytic continuation for $\kap>\frac{n}{2}+2$ if $n$ is even and $\chi_S$ is trivial, and for $\kap\geq\frac{n}{2}+2$ otherwise. 

For our purpose, it is more convenient to work in adelic form.  
Put 
\[P_S(\QQ)=\Big\{\mtrx{*}{*}{0}{*}[0,*,*]\in J_S(\QQ)\Big\}. \]

Let $|\;|_\AA$ be the module of the idele group of $\QQ$ and put $\bfe_\AA(x)=\prod_v\bfe_v(x_v)$ for $x=(x_v)_v\in \AA$. 
Let $l_0$ be the characteristic function of $\prod_pL_p$. 
We define the function $l\Pr_0$ on $X_\AA$ by 
\begin{align*}
l\Pr_0(x)&=e^{-\pi S[x_\infty]}l_0(x_\bff) & \tx{for }x&=(x_v)\in X_\AA. 
\end{align*} 
Here, we write $x_\bff$ for the finite part of $x$. 
Put 
\begin{align*}
\wtl{E}_{\kap,S}(g; s)&=\sum_{\gam\in P_S(\QQ)\bsl J_S(\QQ)}\phi_{\kap,s}(\gam g) & \tx{for }g\in J_S(\AA), 
\end{align*}
where $\phi_{\kap, s}=\prod_v\phi_{\kap, s, v}$ is defined by  
\[\phi_{\kap,s}(p[\xi, \eta, \zet]\kap_1\kap_2)
=|t|_\AA^{\kap+s}\bfe_\AA(\zet)l\Pr_0(\xi)(c\iu+d)^{-\kap} \]
for 
\begin{align*}
&p=\mtrx{t}{*}{0}{t^{-1}}\in\SL_2(\AA), & &[\xi, \eta, \zet]\in H_S(\AA), \\
&\kap_1=\mtrx{*}{*}{c}{d}\in\SO_2(\RR), & &\kap_2\in\prod_p\SL_2(\ZZ_p)H_S(\ZZ_p). 
\end{align*}
Then it is immediate that
\begin{align*}
\wtl{E}_{\kap,S}(g; 0)&=E_{\kap,S}(g(\iu, 0))j_\kap(g, (\iu, 0))^{-1} &
\tx{for }g&\in J_S(\RR). 
\end{align*} 

For each place $v$ of $\QQ$, we put 
\begin{multline*}
I_{\kap, a, \alp, v}(g_v; s)=\\ 
\int_{\QQ_v\times X_v}\phi_{\kap,s,v}\Big([u,0,0]\mtrx{0}{-1}{1}{w}g_v;s\Big)\bfe_v(-aw-S(\alp,u))dwdu. 
\end{multline*}
Then \cite[Lemma 3.3]{Su} shows that
\[\wtl{E}_{\kap,S}(g; s)=\sum_{a\in\QQ, \;\alp\in X}c_{\kap, a, \alp}(g;s), \]
where
\beq
c_{\kap,a,\alp}(g;s)=\del(a=S[\alp]/2)\phi_{\kap,s}([\alp,0,0]g)
+\prod_vI_{\kap, a, \alp, v}(g_v; s).  \label{tag:81}
\eeq

To write down $I_{\kap, a, \alp, \infty}(g;s)$, we put
\[\ome(z;\lam, \mu)=\vGm(\mu)^{-1}z^\mu\int_0^\infty e^{-zt}(t+1)^{\lam-1}t^{\mu-1}dt\]
for $\mu,\;z\in \frkH\Pr$, letting $\frkH\Pr$ be the right half-plane $\{z\in\CC\;|\; \Re z>0\}$.  
Then $\ome(z; \lam, \mu)$ can be continued as a holomorphic function to the whole $\frkH\Pr\times\CC^2$ and satisfies 
\[\ome(z; \lam, 0)=1\] 
(see \cite[Theorem 3.1, (3.13), (3.15)]{Sh}). Assume that 
\[g=[\xi, \eta,0]\mtrx{y^{1/2}}{xy^{-1/2}}{0}{y^{-1/2}}\in J_S(\RR). \]

The proposition below is easily deduced from \cite[Proposition 3.4]{Su}. 

\begin{proposition}\label{prop:81}
If $a>S[\alp]/2$, then 
\begin{multline*}
I_{\kap, a, \alp, \infty}(g;s)=\frac{(-1)^{\kap/2}2^{\kap-n/2}\pi^{\kap(n,s)}}
{(\det S)^{1/2}\vGm(\kap(n,s))}\ome(2\pi y(2a-S[\alp]); \kap(n,s), s/2)\\
\times(a-S[\alp]/2)^{\kap(n,s)-1}q^a\bfe(S(\alp, \tau\xi+\eta))j_\kap(g, (\iu,0))^{-1}y^{s/2}. 
\end{multline*}
Here we put $\kap(n,s)=\kap+\tfrac{s-n}{2}$. 
\end{proposition}

To represent $I_{\kap,a,\alp,p}$, an auxiliary function $I_p(S, (a, \alp); X)$ is needed. 
For the precise definition, we refer the reader to \cite[(2.21)]{Su}. 
Here we simply record a useful relation  
\beq
I_p(S, (a, \alp);p^{-s})=I_{\kap, a, \alp, p}\bigl(1; s-\kap+1+\tfrac{n}{2}\bigl) \label{tag:82}
\eeq
(see \cite[(3.18)]{Su}). 
We define the polynomial $\vrh_{p, S}$ by  
\[\vrh_{p,S}(X)=\begin{cases}
1-p^{-1}X^2 &\tx{if $p\in\frkS_0$, $2\nmid n$. }\\
1-\eta_p(S)p^{-1/2}X &\tx{if $p\in\frkS_1$, $2\nmid n$. }\\
1-\xi_p(S)p^{-1}X &\tx{if $p\in\frkS_0\cup\frkS_1$, $2|n$. }\\
1 &\tx{if $p\in\frkS_2$. }
\end{cases}\]

From now on, we assume that $(a, \alp)\in \ZZ_p\times L^*_p$ and $a\neq S[\alp]/2$, where $L^*_p$ is the closure of $L^*$ in $X_p$. 
Put 
\[I_{p, S, a, \alp}(X)=\vrh_{p,S}(X)^{-1}I_p(S, (a, \alp);X). \]

\begin{proposition}\label{prop:82}
The function $I_{p,S,a,\alp}$ is given as follows. 
\begin{enumerate}
\renewcommand\labelenumi{(\theenumi)}
\item If $n$ odd, then $I^{\sim}_{p,S,a,\alp}=l_{p,S, \Del_{a,\alp}}$, where 
\[I^\sim_{p,S,a,\alp}(X)
=(1-\Psi_p((-1)^{(n+1)/2}\Del_{a,\alp})p^{-1/2}X)X^{-\frkf_p(\Del_{a,\alp})}I_{p, S, a, \alp}(X). \]
\item If $n$ is even, then $I^{\sim}_{p,S,a,\alp}=l_{p,S, D_{a,\alp}}$, where 
\[I^\sim_{p,S,a,\alp}(X)=X^{-\ord_pD_{a,\alp}}I_{p, S, a, \alp}(X^2). \]
\end{enumerate}
\end{proposition}

\begin{proof}
\begin{table}
\caption{}
\label{tab:81}
\begin{tabular}{c|c|c|c}
$(n_0, \pal)$ & $(n_0\Pr, \pal\Pr)$ & $I^\sim_{p, S, a, \alp}$\\\hline
$(0 ,0 )$ & $(1 ,0 )$ & $l_{2f}$ \\
$(0 ,0 )$ & $(1 ,1 )$ & $l_{2f+1}$ \\
$(1 ,0 )$ & $(0 ,0 )$ & $l_f-p^{-1/2}l_{f-1}$ \\
$(1 ,0 )$ & $(2 ,0 )$ & $l_f+p^{-1/2}l_{f-1}$ \\
$(1 ,0 )$ & $(2 ,1 )$ & $l_f$ \\
$(1 ,1 )$ & $(0 ,0 )$ & $l_f-p^{-1/2}l_{f-1}+p^{1/2}(l_{f-1}-p^{-1/2}l_{f-2})$ \\
$(1 ,1 )$ & $(2 ,1 )$ & $l_f+p^{1/2}l_{f-1}$ \\
$(1 ,1 )$ & $(2 ,2 )$ & $l_{f+1}+p^{-1/2}l_f+p^{1/2}(l_f+p^{-1/2}l_{f-1})$ \\
$(2 ,0 )$ & $(1 ,0 )$ & $l_{2f,-1}$ \\
$(2 ,0 )$ & $(3 ,1 )$ & $l_{2f+1,-1}$ \\
$(2 ,1 )$ & $(1 ,0 )$ & $h_{2f, 1}$ \\
$(2 ,1 )$ & $(1 ,1 )$ & $h_{2f+1, 1}$ \\
$(2 ,1 )$ & $(3 ,1 )$ & $h_{2f+1, -1}$ \\
$(2 ,1 )$ & $(3 ,2 )$ & $h_{2f+2, -1}$ \\
$(2 ,2 )$ & $(1 ,1 )$ & $l_{2f, -1}+pl_{2f-2, -1}$ \\
$(2 ,2 )$ & $(3 ,2 )$ & $l_{2f+1, -1}+pl_{2f-1, -1}$ \\
$(3 ,1 )$ & $(2 ,0 )$ & $l_f+p^{-1/2}l_{f-1}-p^{1/2}(l_{f-1}+p^{-1/2}l_{f-2})$ \\
$(3 ,1 )$ & $(2 ,1 )$ & $l_f-p^{1/2}l_{f-1}$ \\
$(3 ,1 )$ & $(4 ,2 )$ & $l_{f+1}-p^{-1/2}l_f-p^{1/2}(l_f-p^{-1/2}l_{f-1})$ \\
$(3 ,2 )$ & $(2 ,1 )$ & $l_f-pl_{f-2}$ \\
$(3 ,2 )$ & $(2 ,2 )$ & $l_{f+1}+p^{-1/2}l_f-p(l_{f-1}+p^{-1/2}l_{f-2})$ & $f>0$\\
          &           & $l_1+(p^{-1/2}+p^{1/2})l_0$ & $f=0$\\
$(3 ,2 )$ & $(4 ,2 )$ & $l_{f+1}-p^{-1/2}l_f-p(l_{f-1}-p^{-1/2}l_{f-2})$ & $f>0$\\
          &           & $l_1-(p^{-1/2}+p^{1/2})l_0$ & $f=0$\\
$(4 ,2 )$ & $(3 ,1 )$ & $l_{2f}-pl_{2f-2}$ \\
$(4 ,2 )$ & $(3 ,2 )$ & $l_{2f+1}-pl_{2f-1}$
\end{tabular}
\end{table}
In discussing the proof of this result, we explain Sugano's formulation and its relationship with our notation.   
Let $n_0=n_{0,p}$ be the dimension of an anisotropic kernel of $S$. 
Set $L\Pr_p=\{x\in L_p^*\;|\;S[x]/2\in p^{-1}\ZZ_p\}$. 
Note that $L\Pr_p/L_p$ is a vector space over $\FF_p$ and its dimension is denoted by $\pal=\pal_p$. 
We can easily check that $s_p(S)=\pal$, using the description of $L_p$ in the proof of Lemma \ref{lem:52} or \ref{lem:71}. 

Put $S_{a,\alp}=\Left S & S\alp \\ \trs\alp S & 2a \Right$. 
We can choose $(a_0, \alp_0)\in \ZZ_p\times L^*_p$, $x\in L_p$ and a non-negative integer $f$ to satisfy the following conditions:  
\begin{enumerate}
\renewcommand\labelenumi{(\theenumi)}
\item $S_{a,\alp}=S_{a_0,\alp_0}\bigl[\Left 1 & x \\ 0 & p^f\Right\bigl]$;
\item $\ZZ^{n+1}$ is a maximal integral lattice with respect to $S_{a_0,\alp_0}$ 
\end{enumerate}(see \cite[Lemma 2.5]{Su}). 
For simplicity we put $S^\sim=S_{a_0,\alp_0}$. 
We define $n_0\Pr$ and $\pal\Pr$, replacing $S$ by $S^\sim$ in the definitions of $n_0$ and $\pal$ respectively.

Suppose that $n$ is odd. 
Then $\xi_p(S_{a,\alp})=\Psi_p((-1)^{(n+1)/2}\Del_{a,\alp})$. 
It follows from (\ref{tag:11}) that 
\[\Del_{a,\alp}=\det S_{a,\alp}=p^{2f}\det S^\sim=p^{2f}\frkd_{S^\sim}d_{S^\sim}^2, \]
and hence  
\beq
\frkf_p(\Del_{a,\alp})=f+[\pal\Pr/2]. \label{tag:83}
\eeq
If $n$ is even, then 
\[d_SD_{a,\alp}=(\det S_{a,\alp})/2=p^{2f}(\det S^\sim)/2=p^{2f}b_{S^\sim}d_{S^\sim}^2, \]
and hence 
\beq
\ord_pD_{a,\alp}=2f-[\pal/2]+\pal\Pr. \label{tag:84}
\eeq 

Sugano lists up explicit formulas of $I_p(S, (a,\alp);X)$ for all possible pairs of $(n_0, \pal)$, $(n_0\Pr, \pal\Pr)$ in \cite[Proposition 2.14 (\roman{two})]{Su}, which combined with a case by case calculation shows that $\til{I}_{p,S,a,\alp}$ is given by Table \ref{tab:81}. 
Using (\ref{tag:83}) and (\ref{tag:84}), we can verify the desired equality.  
\end{proof}

\begin{corollary}\label{cor:81}
If $\kap$ be even, then $E_{\kap,S}\in J^M_{\kap,S}$, i.e., 
there exists a constant $C_{\kap,S}$ and a function $A:\NN\to\CC$ such that
\[E_{\kap,S}(\tau,w)=\sum_{(a,\alp)\in\calt^0}q^a\bfe(S(\alp,w))
+C_{\kap,S}\sum_{(a,\alp)\in\calt^+}A(D_{a,\alp})q^a\bfe(S(\alp,w)). \] 
Moreover, by putting $k=\kap-\bigl[\tfrac{n+1}{2}\bigl]$, the function $A$ is given as follows.  
\begin{enumerate}
\renewcommand\labelenumi{(\theenumi)}
\item If $n$ is odd, then 
\begin{align*}
A(N)&=L(1-k,\psi_{(-1)^k\del_SN})\frkf_{\del_SN}^{k-1/2}\prod_pl_{p,S,\del_SN}(p^{k-1/2})&\tx{for }N&\in\NN. 
\end{align*}
\item If $n$ is even, then 
\begin{align*}
A(N)&=N^{(k-1)/2}\prod_pl_{p,S,N}(p^{(k-1)/2})&\tx{for }N&\in\NN. 
\end{align*}
\end{enumerate}
\end{corollary}

\begin{proof}
Note that $I_{k, a, \alp, \infty}(g;0)=0$ (see \cite[Proposition 3.4 (\roman{one})]{Su}).  
Corollary \ref{cor:81} easily follows from (\ref{tag:81}), (\ref{tag:82}), Proposition \ref{prop:81} and \ref{prop:82}.   
\end{proof}

%%%%%%%%%%%%%%%%%%%%%%%%%%%%%%%%%%%%%%%%%%%%%%%%%%%%%%%%%%%%%%%%%%%%%%%%%

\section{\bf Fourier-Jacobi coefficients of Siegel Eisenstein series}\label{sec:9}

The symplectic group $\Sp_\ell$ is an algebraic group defined over $\QQ$, the group of $D$-valued points of which is given by 
\[\Sp_\ell(D)
=\Big\{\alp\in\GL_{2\ell}(D)\;|\;\trs\alp\mtrx{0}{-\ono_\ell}{\ono_\ell}{0}\alp
=\mtrx{0}{-\ono_\ell}{\ono_\ell}{0}\Big\} \]
for every $\QQ$-algebra $D$. The archimedean part $\Sp_\ell(\RR)$ of $\Sp_\ell$ acts transitively on Siegel upper half-space $\frkH_\ell$ by $\alp Z=(aZ+b)(cZ+d)^{-1}$ for $\alp=\Left a & b \\ c & d\Right\in\Sp_\ell(\RR)$ and $Z\in\frkH_\ell$. 
We define the automorphy factor $j_\kap$ on $\Sp_\ell(\RR)\times\frkH_\ell$ by $j_\kap(\alp,Z)=\det(cZ+d)^\kap$. 

Recall that the Siegel Eisenstein series 
$E^\ell_\kap$ on $\frkH_\ell$ is defined by
\[E^\ell_\kap(Z)=\sum_{\{C,D\}}\det(CZ+D)^{-\kap}, \]
where $\{C, D\}$ runs over a complete set of representatives of the equivalent classes of symmetric coprime pairs of degree $\ell$. 

Let $S_\ell(\ZZ)$ (resp. $T^+_\ell$) be the set of integral symmetric (resp. positive definite symmetric half-integral) matrices of size $\ell$. 
As is well-known, the $h$-th Fourier coefficient of $E^\ell_\kap$ is equal to 
\beq
(2(\iu\pi)^\kap)^\ell\vGm_\ell(\kap)^{-1}\det(2h)^{\kap-(\ell+1)/2}\prod_pb_p(h,\kap). \label{tag:91}
\eeq
for $h\in T^+_\ell$. Here we put $\Gam_\ell(s)=\pi^{\ell(\ell-1)/4}\prod_{i=0}^{\ell-1}\vGm(s-i/2)$. 
Recall that the Siegel series $b_p(h,s)$ for $h$ is defined by  
\[b_p(h,s)=\sum_{\alp\in S_\ell(\QQ_p)/S_\ell(\ZZ_p)}\bfe_p(-\tr(h\alp))\nu(\alp)^{-s}, \]
where we put $S_\ell(R)=S_\ell(\ZZ)\otimes_\ZZ R$ for a ring $R$ and define
$\nu(\alp)$ to be the product of the denominator ideals of fundamental divisors of $\alp$. 

We note that there exists polynomial $f_p(h;X)$ such that 
\[f_p(h;p^{-s})=b_p(h,s).\]  
Put
\[\gam_{p,h}(X)=\begin{cases} 
\dfrac{(1-X)\prod_{j=1}^{\ell/2}(1-p^{2j}X^2)}{1-\xi_p(h)p^{\ell/2}X} &\tx{if $2|\ell$. }\\
(1-X)\prod_{j=1}^{\strut{(\ell-1)/2}}(1-p^{2j}X^2)&\tx{if $2\nmid \ell$. }
\end{cases}\]
Then there exists a polynomial $F_p(h;X)$ such that 
\[f_p(h;X)=\gam_{p,h}(X)F_p(h;X). \] 

On the other hand, let $F$ be a Siegel modular form of weight $\kap$ with respect to $\Sp_{n+1}(\ZZ)$. 
The $S/2$-th Fourier-Jacobi coefficient $F_{S/2}$ of $F$ is defined by
\[F_{S/2}(\tau,w)=\sum_{(a,\alp)\in\calt^0\cup\calt^+}
c_F(2^{-1}S_{a,\alp})q^a\bfe(S(\alp,w)). \]
As is well-known, $F_{S/2}\in J_{\kap,S}$. 

The following proposition is the principal result of this section, which connects the Siegel Eisenstein series with the Jacobi Eisenstein series. 

\begin{proposition}\label{prop:91}
The $S/2$-th Fourier-Jacobi coefficient of $E^{n+1}_\kap$ is equal to 
\[(2(\iu\pi)^\kap)^n\vGm_n(\kap)^{-1}
(\det S)^{\kap-(n+1)/2}E_{\kap,S}(\tau,w)\prod_pb_p(2^{-1}S,\kap). \]
\end{proposition}

\begin{proof}
To prove this result, it is again best to consider Eisenstein series on the adele group. 
Put $P_\ell=\{\Left * & * \\ \bf0_\ell &*\Right\in\Sp_\ell\}$ and $\bfi=\iu\ono_\ell\in\frkH_\ell$. 
The standard maximal compact subgroup $C_\ell$ of $\Sp_\ell(\AA)$ is defined by
\begin{align*}
C_\ell&=\prod_vC_{\ell,v}, &
C_{\ell,v}&=\begin{cases}
\{x\in \Sp_\ell(\RR)\;|\;x(\bfi)=\bfi\} &\tx{if $v=\infty$. } \\
\Sp_\ell(\QQ_p)\cap\GL_{2\ell}(\ZZ_p) &\tx{if $v=p$. }
\end{cases}
\end{align*}

We define the series $\wtl{E}^\ell_\kap(x;s)$ by 
\begin{align*}
\wtl{E}^\ell_\kap(x;s)&=\sum_{\gam\in P_\ell(\QQ)\bsl\Sp_\ell(\QQ)}\vep_{\kap,s}(\gam x) &
\tx{for }x&\in \Sp_\ell(\AA). 
\end{align*}
Here $\vep_{\kap,s}(x)=\prod_v\vep_{\kap,s,v}(x_v)$ is defined by
\[\vep_{\kap,s,v}(x_v)=|\det a_v|^{\kap+s}\times 
\begin{cases}
j_\kap(w_\infty,\bfi)^{-1}&\tx{if $v=\infty$. }\\
1&\tx{if $v\neq\infty$. }
\end{cases}\]
for $x=pw\in P_\ell(\AA)C_\ell$ with $p=\Left a & * \\ 0 & \trs a^{-1}\Right$. Then 
\begin{align*}
\wtl{E}^\ell_\kap(g;0)&=E^\ell_\kap(g(\bfi))j_\kap(g,\bfi)^{-1} &
&\tx{for }g\in\Sp_\ell(\RR). 
\end{align*}

The $S/2$-th Fourier-Jacobi coefficient of $\wtl{E}^{n+1}_\kap$ is given by 
\[\wtl{E}_{S/2}(g;s)=\int_{S_n(\QQ)\bsl S_n(\AA)}\wtl{E}^{n+1}_\kap(\tau(z)g; s)\bfe_\AA(-\tr(Sz)/2)dz \]
for $g\in J_S(\AA)$, where put $\tau(z)=\mtrx{\ono_{n+1}}{z}{0}{\ono_{n+1}}$ and embed $S_n(\AA)$ in the upper left $n\times n$ block of $S_{n+1}(\AA)$. 
Put 
\begin{align*}
\La x,y,z\Ra&=\begin{pmatrix}
\ono_n & x & z-x\trs y & y \\
       & 1 & \trs y    &   \\
       &   & \ono_n    &   \\
       &   & -\trs x   & 1
\end{pmatrix}, 
& \bar\alp&=
\begin{pmatrix}
\ono_n &   &         &   \\
       & a &         & b \\
       &   & \ono_n  &   \\
       & c &         & d
\end{pmatrix}
\end{align*}
for $x$, $y\in X$, $z\in S_n(\QQ)$ and $\alp=\Left a & b \\ c & d\Right\in\SL_2(\QQ)$, and let 
\begin{align*}
J\Pr_S&=\{\bar\alp\cdot\La x,y,z\Ra\}, & \ker S&=\{\La 0,0,z\Ra \;|\;\tr(Sz)=0\}. 
\end{align*}
Note that the quotient group $J\Pr_S/\ker S$ is the Jacobi group $J_S$. 
Using this identification, we can regard $\wtl{E}_{S/2}(g;s)$ as a function on $J_S(\AA)$ in view of $\wtl{E}_{S/2}(\La0,0,z\Ra g;s)=\bfe_\AA(\tr(Sz)/2)\wtl{E}_{S/2}(g;s)$. 

It is shown in the proof of \cite[Theorem 3.2]{Ik} that  
\[\wtl{E}_{S/2}(g;s)=\sum_{\gam\in P_S(\QQ)\bsl J_S(\QQ)}\phi\Pr_{\kap,s}(\gam g)\]
where $\phi\Pr_{\kap,s}(g)=\prod_v\phi\Pr_{\kap,s,v}(g_v)$ is defined by  
\begin{align*}
\phi\Pr_{\kap,s,v}(g_v)
&=\int_{S_n(\QQ_v)}\vep_{k,s,v}(\xi_0\tau(z)g_v)\bfe_v(-\tr(Sz)/2)dz, \\
\xi_0&=\begin{pmatrix}         
              &   & -\ono_n &   \\ 
              & 1 &         &   \\
\ono_n        &   &         &   \\
              &   &         & 1 
\end{pmatrix}. 
\end{align*} 

Observe that 
\[\phi\Pr_{\kap,s,v}(p[x,y,z]w)=|t|_v^{\kap+s}\bfe_v(z)\phi\Pr_{\kap,s,v}(x)\times
\begin{cases}
(c\iu+d)^{-\kap} &\tx{if $v=\infty$ }\\
1&\tx{if $v\neq\infty$ }
\end{cases}\] 
for $p=\Left t & * \\ 0 & t^{-1}\Right\in \SL_2(\QQ_v)$, $[x,y,z]\in H_S(\QQ_v)$ and $w=\Left * & * \\ c & d \Right\in C_{1,v}$. 

We can easily observe that
\begin{align*}
\phi\Pr_{\kap,s, p}([x,0,0])&=\phi\Pr_{\kap,s,p}\Big([x,0,0][0,y,0]\mtrx{1}{b}{0}{1}\Big)\\
&=\bfe_p(S(x,y)+bS[x]/2)\phi\Pr_{\kap,s,p}([x,0,0])
\end{align*}
for $y\in\ZZ^n_p$ and $b\in\ZZ_p$. 
As $\ZZ^n_p$ is a maximal lattice with respect to $S$, it follows that
\[\phi\Pr_{\kap,s,p}([x,0,0])=\del(x\in\ZZ^n_p)b_p(2^{-1}S,\kap+s). \]
We thus conclude that
\begin{align*}
\phi\Pr_{\kap,s,p}(g)&=\phi_{\kap,s,p}(g)b_p(2^{-1}S,\kap+s) & &\tx{for }g\in J_S(\QQ_p) 
\end{align*} 
(see \S \ref{sec:8} for the definition of $\phi_{\kap,s,p}$).  
It is immediate that
\[\phi\Pr_{\kap,s,\infty}([x,0,0])=e^{-\pi S[x]}\Xi\bigl(\ono_n, 2^{-1}S;\kap+\tfrac{s}{2},\tfrac{s}{2}\bigl), \]
where the confluent hypergeometric function $\Xi(Y, S; s, s\Pr)$ is defined by
\[\int_{S_n(\RR)}\bfe(-2^{-1}\tr(SX))\det(X+\iu Y)^{-s}\det(X-\iu Y)^{-s\Pr}dX\]
for $s$, $s\Pr\in\CC$ and $0<Y\in S_n(\RR)$. Since
\[\Xi(\ono_n, 2^{-1}S; \kap, 0)
=(2(\iu\pi)^\kap)^n\vGm_n(\kap)^{-1}(\det S)^{\kap-(n+1)/2}e^{-\pi\tr(S)}, \] 
we obtain 
\[\phi\Pr_{\kap,s,\infty}(g)
=(2(\iu\pi)^\kap)^n\vGm_n(\kap)^{-1}(\det S)^{\kap-(n+1)/2}e^{-\pi\tr(S)}\phi_{\kap,s,\infty}(g). \]
All of these results combine to give Proposition \ref{prop:91}.  
\end{proof}

\begin{lemma}\label{lem:91}
We abbreviate $s=s_p(S)$, $\xi=\xi_p(S)$ and $\eta=\eta_p(S)$. 
\begin{enumerate}
\renewcommand\labelenumi{(\theenumi)}
\item If $n$ is odd, then $f_p(2^{-1}S;X)$ equals
\[\begin{cases}
(1-X)(1+\eta p^{(n+1)/2}X)\prod_{j=1}^{(n-1)/2}(1-p^{2j}X^2) &\tx{if $s=1$. }\\
(1-X)\prod_{j=1}^{\strut{(n+s-1)/2}}(1-p^{2j}X^2) &\tx{if $s\neq1$. }
\end{cases}\]
\item If $n$ is even, then $f_p(2^{-1}S;X)$ equals
\[\begin{cases}
(1-X)\prod_{j=1}^{n/2}(1-p^{2j}X^2) &\tx{if $s=1$. }\\
(1-X)(1+(-1)^{s/2}\xi p^{(n+s)/2}X)\prod_{j=1}^{\strut{(n+s)/2-1}}(1-p^{2j}X^2) &\tx{if $s\neq1$. }
\end{cases}\]
\end{enumerate}
\end{lemma}

\begin{proof}
Let $B$, $B_2$, $n$ and $l$ be the ones in \cite[Lemma 3.3]{Ka}. 
From the proof of Lemma \ref{lem:52} or \ref{lem:71}, we can observe that
\[\xi_p(B_2)=\begin{cases}
\eta_p(B) &\tx{if $l=n-1$, $2\nmid n$. }\\
-\xi_p(B) &\tx{if $l=n-2$, $2|n$. } 
\end{cases}\]
Therefore our assertion is a restatement of \cite[Lemma 3.3]{Ka}.  
\end{proof}

We put 
\[\wtl{F}_p(h;X)=\begin{cases}
X^{-\frkf_p(\det(2h))}F_p(h;p^{-(\ell+1)/2}X) &\tx{if $2|\ell$. }\\
X^{-\ord_p(2^{-1}\det(2h))}F_p(h;p^{-(\ell+1)/2}X^2) &\tx{if $2\nmid\ell$. }
\end{cases}\]
\begin{proposition}\label{prop:92}
Let $(a,\alp)\in\calt^+$. Then 
\[\wtl{F}_p(2^{-1}S_{a,\alp};X)=\begin{cases}
l_{p,S,\Del_{a,\alp}}(X)&\tx{if $2\nmid n$. }\\
l_{p,S,D_{a,\alp}}(X)&\tx{if $2|n$, $s_p(S)\neq 2$. }\\
l_{p,S,D_{a,\alp}}(X)(X^{-1}-\xi_p(S)X) &\tx{if $2|n$, $s_p(S)=2$. }
\end{cases}\]
\end{proposition}

\begin{proof}
To be uniform, we denote the right-hand side by $G_p(X)$. 
From Lemma \ref{lem:91}, we have 
\[\frac{f_p(2^{-1}S;X)\vrh_{p,S}(p^{n/2+1}X)}{\gam_{p,2^{-1}S_{a,\alp}}(X)}=\begin{cases}
1-\xi_p(S_{a,\alp})p^{(n+1)/2}X &\tx{if $2\nmid n$. }\\
1&\tx{if $p\notin\frkS_2$, $2|n$. }\\
1-\xi_p(S)p^{(n+2)/2}X &\tx{if $p\in\frkS_2$, $2|n$. }
\end{cases} \]
Observe that 
\begin{multline*}
\frac{(2(\iu\pi)^\kap)^{n+1}}{(2(\iu\pi)^\kap)^n}\frac{\vGm_{n+1}(\kap)^{-1}}{\vGm_n(\kap)^{-1}}
\frac{(\det S_{a,\alp})^{\kap-(n+2)/2}}{(\det S)^{\kap-(n+1)/2}}\\
=\frac{(-1)^{\kap/2}2^{\kap-n/2}\pi^{\kap-n/2}(a-S[\alp]/2)^{\kap-n/2-1}}{(\det S)^{1/2}\vGm(\kap-n/2)}
=I_{\kap,a,\alp,\infty}(\ono_2,0). 
\end{multline*}
Combining (\ref{tag:91}), Proposition \ref{prop:82} and \ref{prop:91}, we have 
\[\prod_p\wtl{F}_p(2^{-1}S_{a,\alp};p^{k\Pr-1/2})=\prod_pG_p(p^{k\Pr-1/2}), \]
where $k\Pr=k$ or $k\Pr=k/2$ according as $n$ is odd or even. 
We have
\[\prod_p\wtl{F}_p(2^{-1}S_{a,\alp};X_p)=\prod_pG_p(X_p)\]
by \cite[Lemma 10.1]{Ik2}. 
As the constant term of $F_p(2^{-1}S_{a,\alp};X)$ is $1$, we obtain the desired fact. 
\end{proof}
 
In \cite{Ik2} Ikeda proved the following theorem. 
\begin{theorem}[Ikeda]\label{thm:91}
Suppose that $n$ is odd and $\kap=k+\tfrac{n+1}{2}$ is even. 
Let $f\in S_{2k}(\SL_2(\ZZ))$ be a normalized Hecke eigenform, the $L$-function of which is given by 
\[\prod_p(1-\alp_pp^{k-1/2-s})^{-1}(1-\alp_p^{-1}p^{k-1/2-s})^{-1}. \]
Let $g\in S^+_{k+1/2}(4)$ be a Hecke eigenform corresponding to $f$ under the Shimura correspondence. 
Put $D_h=2^{n+1}\det h$ for $h\in T_{n+1}^+$. 
Then the function 
\begin{align*}
F(Z)&=\sum_{h\in T^+_{n+1}}c(\frkd_{D_h})\frkf_{D_h}^{k-1/2}\prod_p\wtl{F}_p(h;\alp_p)\bfe(\tr(hZ)), & 
Z&\in\frkH_{n+1}
\end{align*}
is a Siegel cusp form of weight $\kap$ with respect to $\Sp_{n+1}(\ZZ)$. 
\end{theorem}

Now we have the following important corollary. 

\begin{corollary}\label{cor:91}
With the notation of Theorem \ref{thm:91}, the $S/2$-th Fourier-Jacobi coefficient of $F$ coincides with the image of $f$ under the lifting described in Theorem \ref{thm:31}. 
\end{corollary}

%%%%%%%%%%%%%%%%%%%%%%%%%%%%%%%%%%%%%%%%%%%%%%%%%%%%%%%%%%%%%%%%%%%%%%%%%

\section{\bf Application to Maass spaces on orthogonal groups}\label{sec:10}

We shall give an explicit Fourier coefficient formula of the theta lifting attached to the orthogonal group of signature $(2,n+2)$.  
Put
\begin{align*}
Q&=\Let & & 1 \\ & -S & \\ 1 & & \Rit, &
Q\Pr&=\Let & & 1 \\ & Q & \\ 1 & & \Rit, \\
U&=\QQ e\oplus X\oplus \QQ f, & 
V&=\QQ e\Pr\oplus U\oplus \QQ f\Pr, \\
L_1&=\ZZ e\oplus L\oplus \ZZ f, &
L_2&=\ZZ e\Pr\oplus L_1\oplus \ZZ f\Pr. 
\end{align*}
The special orthogonal group $\SO(V)$ of $Q\Pr$ is an algebraic group defined over $\QQ$, 
the group of $D$-valued points of which is given by 
\[\SO(V)(D)=\{\alp\in \SL_{n+4}(D)\;|\; \trs\alp Q\Pr\alp=Q\Pr\} \]
for every $\QQ$-algebra $D$. 

For each place $v$ of $\QQ$, we set $X_v=X\otimes_\QQ\QQ_v$ and $U_v=U\otimes_\QQ\QQ_v$. Put
\begin{align*}
\frkD_S&=\{\calx\in U_\infty \;|\; Q(e+f, \calx)>0,\; Q[\calx]>0\}, \\
\cald_S&=\{\calz=\calx+\iu\caly\in U_\infty\otimes_\RR\CC\;|\; \calx\in U_\infty,\; \caly\in \frkD_S\}. 
\end{align*}
The connected component of $\ono_{n+4}$ in the topological group $\SO(V)(\RR)$ is denoted by $\SO(V)(\RR)^\circ$. 
The action of the group $\SO(V)(\RR)^\circ$ on $\cald_S$ and the automorphy factor $j(g, \calz)$ on $\SO(V)(\RR)^\circ\times\cald_S$ are defined by 
\begin{align*}
g\calz^\sim&=(g\calz)^\sim j(g, \calz), & \calz^\sim&=\Let -Q[\calz]/2 \\ \calz \\ 1 \Rit 
\end{align*}
for $g\in\SO(V)(\RR)^\circ$ and $\calz\in\cald_S$. 
The modular group $\Tht$ is an arithmetic subgroup $\{\gam\in\SO(V)(\RR)^\circ\;|\;\gam L_2\subset L_2\}$.  
We put 
\begin{gather*}
T=Q^{-1}L_1, \quad\quad\quad\quad\quad T^+=T\cap\frkD_S, \\
T^0=\{\eta\in T\;|\;Q(e+f,\eta)> 0,\; Q[\eta]=0\}.  
\end{gather*}

For a $\CC$-valued function $F$ on $\cald_S$, we put
\[F|_\kap\alp(\calz)=j(\alp, \calz)^{-\kap}F(\alp\calz)\]
for $\alp\in \SO(V)(\RR)^\circ$. 
A holomorphic function $F$ on $\cald_S$ is called a modular (resp. cusp) form of weight $\kap$ if $F|_\kap\gam=F$ for every $\gam\in\Tht$ and has a Fourier expansion of the form  
\[F(\calz)=\sum_{\eta}c_F(\eta)\bfe(Q(\eta, \calz)) \]
where $\eta$ extends over all elements of $\{0\}\cup T^0\cup T^+$ (resp. $T^+$). 
We denote the space of modular forms of weight $\kap$ by $M_\kap(\Tht)$ and that of cusp forms of weight $\kap$ by $S_\kap(\Tht)$. 

Recall that the first Fourier-Jacobi coefficient of $F$ is defined by
\[\phi(\tau,w)=\sum_{(a,\alp)\in\calt^0\cup\calt^+}c_F(e+\alp+af)q^a\bfe(S(\alp,w)). \] 
As is well-known, $\phi\in J_{\kap,S}$. 

For $\eta\in T^0\cup T^+$, we put $D_\eta=D_SQ[\eta]/2$ and
\[\eps(\eta)=\max\{N\in\NN\;|\;N^{-1}\eta\in T\}. \]

\begin{definition}\label{def:101}
A form $F\in M_\kap(\Tht)$ is an element of $M^M_\kap(\Tht)$ if there exists a function $c:\NN\cup\{0\}\to\CC$ such that all $\eta\in T^0\cup T^+$ satisfy
\[c_F(\eta)=\sum_{d|\eps(\eta)}d^{\kap-1}c(d^{-2}D_\eta). \]
Put $S^M_\kap(\Tht)=M^M_\kap(\Tht)\cap S_\kap(\Tht)$. 
\end{definition}
  
\begin{theorem}[Sugano]\label{thm:101}
The association
\[\sum_{(a,\alp)\in\calt^+}c(D_{a,\alp})q^a\bfe(S(\alp,w))\mapsto\sum_{\eta\in T^+}\sum_{d|\eps(\eta)}d^{\kap-1}c(d^{-2}D_\eta)\bfe(Q(\eta, \calz))\]
gives an isomorphism of $J^{\cusp,M}_{\kap,S}$ onto $S^M_\kap(\Tht)$. 
\end{theorem}
\begin{proof}
See \cite[Theorem 6.2, Remark 6.5]{Su} or \cite{MS}. 
\end{proof}

Combining results in \S \ref{sec:3} with Theorem \ref{thm:101}, we obtain the following two theorems. 

\begin{theorem}\label{thm:102}
Under the same notation as in Theorem \ref{thm:31}, we put 
\[c_F(N)=\sum_{d|\eps(\eta)}2^{-\bfb_{bd}(N)}d^{n/2}c_g(\frkd_N)\frkf_N^{k-1/2}\prod_pl_{p,S,d^{-2}N}(\alp_p) \]for each positive integer $N$. 
Putting $\Del_\eta=\del_SD_\eta$, we define the function $F$ on $\cald_S$ by   
Then 
\[F(\calz)=\sum_{\eta\in T^+}c_F(\Del_\eta)\bfe(Q(\eta,\calz)). \]
Then $F$ is a Hecke eigenform in $S^M_\kap(\Tht)$. 
Moreover, the lifting $f\mapsto F$ gives a bijective correspondence, up to scalar multiple, between Hecke eigenforms in 
\[\oplus_{b,d\geq 1,\;b|b_S,\; d|d_S}\{f\in S^\new_{2k}(\Gam_0(bd))\;|\;f|W(p)=\eta_p(S)f\tx{ for each }p|b\}\] 
and those in $S^M_\kap(\Tht)$. 
\end{theorem}

\begin{theorem}\label{thm:103}
Under the same notation as in Theorem \ref{thm:32}, we put 
\[c_F(N)=\sum_{d|\eps(\eta)}d^{n/2}N^{k-1/2}\prod_pl_{p,S,d^{-2}N}(\alp_p)\]
for each positive integer $N$. 
Then 
\[F(\calz)=\sum_{\eta\in T^+}c_F(D_\eta)\bfe(Q(\eta,\calz)) \]
 is a Hecke eigenform in $S^M_\kap(\Tht)$. 
 Moreover, the lifting $f\mapsto F$ defines a bijection from the orbits of Atkin-Lehner involutions in 
\[\bigcup_{d\geq 1, \; d|d_S}\Prm^*_k(\frkd_Sd,\chi_S) \]
and Hecke eigenforms in $S^M_\kap(\Tht)$. 
\end{theorem}

%%%%%%%%%%%%%%%%%%%%%%%%%%%%%%%%%%%%%%%%%%%%%%%%%%%%%%%%%%%%%%%%%%%%%%%%%

\end{document}